\documentclass[11pt]{amsart}
\usepackage{amscd}
\usepackage{amsmath, amssymb}
\usepackage{amsfonts}
\usepackage{enumerate}
\usepackage{esint}
\usepackage{array}
\usepackage{color}
\newcommand{\vp}{\varphi}
\newcommand{\ve}{\varepsilon}
\newcommand{\ddbar}{\sqrt{-1} \partial \overline{\partial}}

\newcommand{\de}{\partial}
\newcommand{\re}{\text{Re}}
\newcommand{\si}{\sqrt{-1}}

\newcommand{\ri}{\rightarrow}

\newcommand{\Ric}{\text{Ric}}

\newcommand{\ov}[1]{\overline{#1}}

\newcommand{\nbi}{\nabla_{i}}

\newcommand{\nbbi}{\nabla_{\bar{i}}}

\newcommand{\ti}{\tilde}

\begin{document}
	\newcounter{remark}
	\newcounter{theor}
	\setcounter{remark}{0}
	\setcounter{theor}{1}
	\newtheorem{claim}{Claim}
	\newtheorem{theorem}{Theorem}[section]
	\newtheorem{lemma}[theorem]{Lemma}
	\newtheorem{corollary}[theorem]{Corollary}
	\newtheorem{proposition}[theorem]{Proposition}
	\newtheorem{question}{question}[section]
	\newtheorem{defn}{Definition}[theor]
	\newtheorem{remark}[theorem]{Remark}
	\newtheorem{conj}[theorem]{Conjecture}
	\numberwithin{equation}{section}

	\title[The form-type Calabi-Yau equation]{The form-type Calabi-Yau equation on a class of complex manifolds}
	\author{Liding Huang}
	\address{School of Mathematical Sciences, Xiamen University, Xiamen 361005, P. R. China}
	\email{huangliding@xmu.edu.cn}

	
	\begin{abstract}
		In this paper, we study the form type Calabi-Yau equation. We define the astheno-Ricci curvature and prove that there exists a solution for the form type Calabi-Yau equation if the astheno-Ricci curvature is non-positive.
	\end{abstract}
	
	\maketitle

	\section{Introduction}
		Let $M$ be a compact Hermitian manifolds of complex dimensiona $n$ with a Hermitian metric $\alpha$.  We use
		\begin{equation}
			H^{1,1}_{BC}(M,\mathbb{R})=\frac{\{\text{d-closed real (1,1)-form}\}}{\{\ddbar \psi: \psi\in C^{\infty}(M, \mathbb{R})\}}
		\end{equation}
		to denotes the Bott-Chern cohomology group and $\text{Ric}(\alpha)$ is the Chern Ricci form of $\alpha$. The first Bott-Chern class is defined by $c_{1}^{BC}=[\text{Ric}(\alpha)]\in H_{BC}^{1,1}(M,\mathbb{R})$. Yau \cite{Yau78} proved Calabi conjecture \cite{Calabi57} on K\"ahler manifolds that have many applications in geometry and physics. We define a  Calabi-Yau manifold to be a compact complex manifold M with $c_{1}^{BC} = 0$ in $H^{1,1}_{BC}(M,\mathbb{R})$ \cite{TV}.  Motivated by mathematics and mathematical physics, the non-K\"ahler Calabi-Yau manifolds has been extensively researched \cite{TY,TV,Yau21}. It is also related to  Reid fantasy \cite{R87}. An important  problem is to find  an analog of Ricci-flat K\"ahler metrics on non-Kähler Calabi-Yau manifolds \cite{TV,Yau21}.

	Tosatti-Weinkove \cite{TW10b} proved that the existence of solutions for complex Monge-Amp\`ere equations on Hermitian manifolds, which implies  a version of Calabi-Yau type theorem for Hermitian metric (pluriclosed metrics). Sz\'ekelyhidi-Tosatti-Weinkove \cite{STW} proved that
	every compact Calabi-Yau manifold admits a Gauduchon metric whose Chern-Ricci curvature is zero, by solving the Monge-Amp\`ere type equation for (n-1)-plurisubharmonic functions.

A Hermitian metric $\omega_{0}$ is balanced if $d(\omega_0^{n-1})=0$. There are a large class
of balanced manifolds \cite{HN81, FLY12, GA66}. The balanced condition exhibits favorable mathematical properties.	A sharp characterization for  balanced metrics was  gave 
	 by \cite{MM}, and the balanced metrics  can be utilized to explore  birational
		geometry since the existence of a balanced
		metric is invariant under birational transformations \cite{AB}. Fu-Li-Yau \cite{FLY12}  proved that suppose  $Y\rightarrow\underline{Y} \rightsquigarrow X_t $ be a conifold transition, then $X_{t}$ is  a balanced manifolds for $t$ small enough, where Y is  a K\"ahler Calabi-Yau manifolds and $\underline{Y}$ is a singular variety with ODP singularities.

		It is natural to seek  a balanced metric 
		which is Chern-Ricci flat  on Calabi-Yau manifolds with a balanced metric. This problem is also related to the Strominger system. Let $\Omega$ be an no-where vanishing holomorphic  $(n, 0)$ form.  Fu-Wang-Wu \cite{FWW10}  want to seek a balanced metric $\hat{\omega}$ satisfying $\|\Omega\|_{\hat{\omega}}=1$. It is equal to find a  balanced metric which is Chern-Ricci flat.  Popovici \cite{PD} also considered this problem. Fu-Wang-Wu \cite{FWW10} proposed the following  approach to study this problem.  Let $\omega_0$ be a balanced metric, then it is suffice to find a metric $\hat{\omega}^{n-1}=\omega_0^{n-1}+\ddbar(\vp\alpha^{n-2})$ satisfying the equation
		\begin{equation}\label{intro equation}
			\hat{\omega}^n=e^{h+b}\alpha^n,
		\end{equation}
		where $\alpha$ is  a Hermitian metric. This equation called a “form-type Calabi-Yau equation" in \cite{FWW10}. It 
		can be reduced to the Monge-Am\`ere type equation for (n-1)-plurisubharmonic functions.
		If there exists a Hermitian metric has non-negative orthogonal bisectional curvature, Fu-Wang-Wu \cite{FWW15} solved this equation on K\"ahler manifolds.  In K\"ahler case, this equation was solved by Tosatti-Weinkove \cite{TW10b}. If $\alpha$ is astheno-K\"ahler, there exists a solution for \eqref{intro equation} using Sz\'ekelyhidi-Tosatti-Weinkove's results \cite{STW}.
		
For the the balanced metric which is Chern-Ricci flat and the equation  \eqref{intro equation},   the following two conjectures were stated in \cite{TV}.
	\begin{conj}\label{conj 1}
		Let M be a compact complex manifold with $c_1^{BC}=0$ and with a balanced metric $\omega_0$. Then there is a balanced metric $\hat{\omega}$ with $[\hat{\omega}^{n-1}]=[\omega_0^{n-1}]$ in $H^{2n-2}(M, R)$ with $Ric(\hat{\omega})=0$.
	\end{conj}

\begin{conj}\label{conj 2}
	Let M be a compact complex manifold with a balanced metric $\omega_{0}$ and a Hermitian metric $\alpha$. Let $h$ be a smooth function. Then there exists a constant $b$ and a function $\vp$ satisfying 
	\begin{equation}\label{n-1 CY equation}
		\begin{cases}
				\hat{\omega}^n=e^{h+b}\alpha^n;\\
				\hat{\omega}^{n-1}=\omega_0^{n-1}+\ddbar(\vp\alpha^{n-2})>0.
			\end{cases}
	\end{equation}
	\end{conj}
Jost-Yau \cite{JY93} defined the astheno-K\"ahler metric $\ddbar \omega^{n-2}=0$.
Following Jost-Yau's definition, we introduce the following definition:
\begin{defn}
	Let $\alpha$ be a Hermitian metric and $*$ be the Hodge-star operator of $\alpha$. We define the astheno-Ricci curvature of  $\alpha$ by
		\begin{equation*}\label{astheno nonpositive}
	*\ddbar (\alpha^{n-2}).
	\end{equation*} 

We say the astheno-Ricci curvature of  $\alpha$ is non-positive if 
	\begin{equation*}
*\ddbar (\alpha^{n-2})\leq 0.
\end{equation*} 
	\end{defn}
	 \begin{remark}
	 	If the astheno-Ricci curvature of  $\alpha$ is equal to 0, it is just an astheno-K\"ahler metric.   The astheno-Ricci curvature  is related to the Ricci curvature (see Proposition \ref{calculate star}). Recently, George \cite{GM} study the equation \eqref{n-1 CY equation} under  the condition  $\ddbar{\alpha^{n-2}}\leq 0$. It is equal to the astheno-Ricci curvature is non-positive.
	 	\end{remark}

	 	 We confirm the conjecture \ref{conj 2} under this condition. More precisely, we have the following theorem
	 	 \begin{theorem}\label{main 2}
	 	 	Let $M$ be a compact complex manifold with dimension n, equipped with Hermitian metrics $\alpha$ and  $\omega_0$. Suppose the astheno-Ricci curvature of $\alpha$ is  non-positive. Given  smooth functions $h$ on $M$, there exists a small  constant $A_0$ such that  for any $0<A\leq A_0$, 
	 	  there exists a unique  pair $(\vp, b)$ where $\vp\in C^{\infty}(M)$  and $b\in \mathbb{R}$ solving
	 	  	\begin{equation}\label{Herm n-1}
	 	  	\begin{cases}
	 	  		\hat{\omega}^n=e^{h+b}\alpha^n;\\
	 	  		\hat{\omega}^{n-1}=\omega_0^{n-1}+\ddbar(\vp\alpha^{n-2})>0, 	\|e^{\vp}\|_{L^{1}(M)}= A,
	 	  	\end{cases}
	 	  \end{equation}
	 	  where $A_0$ is a constant depending on $(M, \alpha),\omega_0$ and $h$.	Here $ \|e^{\vp}\|_{L^{1}(M)}=\int_{M}|e^{\vp}|\,\alpha^n$.	
	 	  Moreover, we have
	 	  \begin{equation}
	 	  	\sup \vp\leq CA,
	 	  	\end{equation} 	
	 	  	where $C$ is a uniform constant depending only on $(M, \alpha),\omega_0$ and $h$.   	 		 
	 	 \end{theorem}
	 	 
	Using Theorem \ref{main 2}, we can prove a Calabi-Yau type theorem. 
	The Bismut connection are researched extensively (see \cite{BJ,FG,ZZ} and the references therein).
	We consider the following Ricci curvature for Bismut connection:
	\begin{equation*}
		\Ric^{B}(\omega)=\Ric(\omega)+dd^{*}\omega.
		\end{equation*} 
		When the metirc is balanced,  the Bismut connection and the Chern connection agree. 
	Tosatti-Weinkove \cite{TW17} proved a Calabi-Yau theorem on $\Ric^{B}$ in balanced class on K\"ahler manifolds. We have the following corollary:
	 \begin{corollary}\label{main 3}
	 	Let M be a compact complex manifold with a balanced metric $\omega_0$. Suppose there exists a Hermitian metric that the astheno-Ricci curvature is non-positive. Then for any $\Phi\in c^{BC}_{1}\in H^{1,1}_{BC}(M, \mathbb{R})$, there is a balanced metric $\hat{\omega}$ with $[\hat{\omega}^{n-1}]=[\omega_0^{n-1}]$ in $H^{2n-2}(M, R)$ with
	 	\begin{equation*} \text{Ric}(\hat{\omega})=\text{Ric}^{B}(\hat{\omega})=\Phi.
	 		\end{equation*}
	 \end{corollary}
	 
	  In \cite{FJS}, Fernandez-Jordan-Streets studied the Busmut Hermitian-Einstein metric (a kind of Ricci curvature of Bismut connection) and proved that it is equal to a Hermitian-Einstein metric on exact holomorphic Courant algebroids. They also gave some examples of compact  manifolds  $X$ with the first Chern class $c_1 = 0 \in H^2(X,Z)$ which do not admit a Bismut Hermitian-Einstein metric. For the Ricci curvature $\Ric^{B}$, we conclude the following corollary:
	  
	  	 \begin{corollary}\label{main 4}
	  	Let M be a compact complex manifold with $c^{BC}_{1}=0$.
	  	  Suppose there exists a  balanced metric $\omega_0$ and a Hermitian metric that the astheno-Ricci curvature is non-positive. Then there is a balanced metric $\hat{\omega}$ with $[\hat{\omega}^{n-1}]=[\omega_0^{n-1}]$ in $H^{2n-2}(M, R)$ with $\text{Ric}(\hat{\omega})=\text{Ric}^{B}(\hat{\omega})=0$.
	  \end{corollary}

	 Since the equation involves the zero order terms,  it is very difficult to solve this equation. To overcome the difficulty, we consider the condition that the  astheno-Ricci curvature of the background metric   is non-positive.  We prove a new $C^0$ estimates for the equation \eqref{n-1 CY equation}. 
	The normalized condition $\int_{M}e^{\vp}\alpha^n=A$  ensures that the positive part $\vp_{+}$ is integrable.  Using Moser iteration, we can prove the solutions are “nearly non-positive". To be more specific, assuming that $\vp\leq 1$, motivated by Fu-Yau \cite{FY}, for a small constant $A$, we can prove 
	 \[\vp\leq CA^{\frac{1}{n+1}}.\]
Then we obtain the above bound for the solutions via
	 the continuity method( see Proposition \ref{upper bound u} and Proposition \ref{claim-2}).
	 
	  The estimates plays an important role in the proof of the second order estimates. The assumption of the astheno-Ricci curvature can be used in the proof of $C^{0}$-estimates and the openness.
	   However, for the second order estimates, there are new difficulties come from $\xi(\vp)$  in the quantity $Q$ for  the terms $\vp B$. We overcome these
	   difficulties by our supremum estimates (see Lemma \ref{Lvp}).
 Adapting Sz\'ekelyhidi’s argument \cite{Szekelyhidi18} based on Blocki's  work \cite{Blocki05, Blocki11}, we established the infimum estimates. 
 We expect that the analogous argument can be
 extended to study the equation \eqref{n-1 CY equation} on general Hermitian manifolds.

	 For the second order estimate, following the idea of \cite{HMW10,Szekelyhidi18,STW} ,  we prove the following second order estimates 
	 \begin{equation}\label{second order estimate 0}
	 \sup_{M}|\ddbar \vp|_{\alpha}\leq C(1+\sup_{M}|\nabla \vp|^2_{\alpha}).
	 \end{equation}
	Since there exists a zero order term in the equation, some essential modifications of the arguments in \cite{STW} are needed. In this paper, we prove
	 	a special $C^{0}$ estimates such that we can use the argument in \cite{STW} to prove \eqref{second order estimate 0}.
	 Then the second order estimate follows from the blowup argument and Liouville type theorem \cite{Szekelyhidi18}.
	 
	 Recently, 	George \cite{GM} also proved that there exists a unique  solution for \eqref{n-1 CY equation} under the non-positive of astheno-Ricci curvature. Compared to his results, our method is
	 quite different. We assume the normalized condition and use a new method to prove the $C^0$ estimates. The proof of $C^{0}$-estimates, the openness and the uniqueness of solutions  heavily depends on the normalized condition. On the other hand, 
	for the second order estimates,   the supremum estimates are used to deal with the bad term which comes form the zero order terms in \eqref{n-1 psh}.

	 The organization of paper is as follows. In Section 2, we will introduce some notations and recall
	 some important properties of the equation \eqref{n-1 CY equation}. The zero order estimate will be established in Section 3.  We will derive the second order estimates and prove the higher order estimates in Section 4, and complete the proof of Theorem \ref{main 2} in Section 5.
	 
\bigskip

{\bf Acknowledgements.} The author
would also like to thank  Professor Jixiang Fu, Jianchun Chu and Rirong  Yuan for many useful discussions.  
	 
	 	\section{Preliminaries and notation}
	 
	Let $*$ be the Hodge star operator with respect to $\alpha$.
	Set 
	\begin{equation*}
		\begin{split}
	&\tilde{\chi}=\frac{1}{(n-1)!}*(\omega_{0}^{n-1}), Z(\partial \vp)=\frac{2}{(n-1)!}*\re(\si \partial \vp \wedge \bar{\partial}(\alpha^{n-2})),\\ &\tilde{B}(\vp)=\frac{1}{(n-1)!}*\ddbar(\alpha^{n-2}).
	\end{split}
	\end{equation*}
	For convenience, we use $Z$ and $\tilde{B}$ to denote  $Z(\partial \vp)$ and $\tilde{B}(\vp)$, respectively. 
	Note 
	\begin{equation*}
	\frac{1}{(n-1)!}*(\ddbar\vp\wedge\alpha^{n-2})=\frac{1}{n-1}(\Delta \vp\alpha-\ddbar \vp).
	\end{equation*}
Using the argument in \cite{TW17,STW},	the equation can be rewritten as
	\begin{equation}\label{n-1 ma}
		\begin{cases}
			(\tilde{\chi}+\frac{1}{n-1}((\Delta \vp)\alpha-\ddbar \vp)+Z+\vp\tilde{B})^{n}=e^{(n-1)(h+b)}\alpha^{n},\\
			\tilde{\omega}=\tilde{\chi}+\frac{1}{n-1}((\Delta \vp)\alpha-\ddbar \vp)+Z+\vp\tilde{B}>0.
		\end{cases}
	\end{equation}
	Let T be the linear map given by
	\begin{equation}\label{trans}
		T(\hat{\omega})=\text{tr}_{\alpha}\, \hat{\omega}\alpha-(n-1)\hat{\omega},
	\end{equation}
	where $\hat{\omega}$ is a real smooth $(1,1)$ form. 
	We also have
		\begin{equation}\label{tran inverse}
	\hat{\omega}=\frac{1}{n-1}((\text{tr}_{\alpha} T(\hat{\omega}))\alpha-T(\hat{\omega})).
	\end{equation}
	Denote
	\begin{equation}\label{notation}
		\omega=T(\tilde{\omega}),\chi=T(\tilde{\chi}), W=T(Z), B=T(\tilde{B}).
	\end{equation}
Then
	\begin{equation}\label{omega formula}
		\omega=\chi+\ddbar \vp+W(\partial \vp)+\vp B. 
	\end{equation}
Suppose
\begin{equation*}
	\begin{split}
		\alpha=\sqrt{-1}\alpha_{i\bar j} dz_i\wedge d z_{\bar j} ,\,\,& \omega=\sqrt{-1}g_{i\bar j} dz_i\wedge d z_{\bar{j}},\\
	\chi=\si \chi_{i\bar j} dz_i\wedge d z_{\bar{j}},\,\,& \tilde{\chi}=\si \tilde{\chi}_{i\bar j}dz_i\wedge d z_{\bar{j}},\\
\tilde{\omega}=\sqrt{-1}\tilde{g}_{i\bar j} dz_i\wedge d z_{\bar{j}},\,\,&	B=\si B_{i\bar j}dz_i\wedge d z_{\bar{j}},\\
\end{split}
\end{equation*}
in local holomorphic coordinate system $(z_1,\cdots, z_n)$. 

	In a neighborhood of $x_0$, we choose normal coordinates $(z_1,\cdots, z_n)$ such that at $x_0$, we have $\alpha_{i\bar j}=\delta_{ij}$, $g_{i\bar j}=\lambda_i\delta_{i\bar j}$ and $\lambda_1\geq \cdots\geq\lambda_n$. 
	Assume that $\tilde{\lambda}_1,\cdots,\tilde{\lambda}_n$ are the eigenvalue of $\tilde{\omega}$ with respect to $\alpha$. 
	Then we have 
	\begin{equation}
	\tilde{\lambda}_i=\frac{1}{n-1}\sum_{k\neq i}\lambda_i
	\end{equation}
	and the equation can be rewritten as
	\begin{equation}\label{n-1 psh}
		\begin{cases}
			F(\omega)=f(\lambda_1,\cdots,\lambda_n)=\log\pi_{k}(\sum_{k\neq i}\lambda_i)=(n-1)(h+b),\\
			\omega \in T(\Gamma_{n}),
		\end{cases}  
	\end{equation}
	where $\Gamma_n$ is the positive orthant.

	Define
	\begin{equation*}
		F^{i\overline{j}}=\frac{\partial F}{\de g_{i\ov{j}}}, \quad
		F^{i\overline{j},k\overline{l}}=\frac{\partial^{2}F}{\de g_{i\ov{j}}\de g_{k\ov{l}}}.
	\end{equation*}
	In addition, the equation \eqref{n-1 psh} also can be viewed as a function on $\tilde{g}$. We denote
	\begin{equation}
		\tilde{F}^{i\bar i}=\frac{\partial F}{\partial \tilde{g}_{i\bar j}}.
		\end{equation}
	Set
	\begin{equation}\label{coef copar}
	 f_k=\frac{\partial f}{\partial \lambda_i}=\frac{1}{n-1}\sum_{k\neq i}\frac{1}{\tilde{\lambda}_i} \text{\,\,~and~\,\, } \tilde{f}_k=\frac{1}{\tilde{\lambda}_{k}}.	
	 \end{equation}
Then 
	\begin{equation}\label{elliptic coefficients}
		\tilde{F}^{i\bar j}=\tilde{f_i}\delta_{ij} \text{\,\,and\,\,} F^{i\bar{i}}=f_i\delta_{ij}.
		\end{equation}
	\subsection{The properties for the form type Calabi-Yau equation}	
	The following properties has been used in \cite{STW}.
	\begin{lemma}\label{prop f}
		\begin{enumerate}
			\item  	$\tilde{\lambda}_1\leq \cdots \leq\tilde{\lambda}_n$.
			\item $\tilde{f}_n\leq \tilde{f}_{n-1}\leq \cdots\leq \tilde{f}_1$.
			\item $f_{k}\leq  \tilde{f}_{1}\leq (n-1)f_k$, $k\geq 2$.
			\item $\tilde{f}_{k}\leq (n-1)f_1,k\geq 2.$
			\item $\frac{1}{n(n-1)} \mathcal{F}\leq F^{k\bar k}$, $k\geq 2$,
		\end{enumerate}
		where $\mathcal{F}=\sum_{i}F^{i\bar i}$.
		\end{lemma}
	\begin{proof}
		By the definition of $\tilde{\lambda}_i$, the (1) and (2) follows. It is easy to see $f_k\leq \frac{1}{\tilde{\lambda}_1}=\tilde{f}_1$. Since $\tilde{\lambda}_i>0$, we have $\tilde{f}_1\leq (n-1)f_k$ and $\tilde{f}_k\leq (n-1)f_1$, $k\geq 2$. Note $\mathcal{F}=\sum_{i}\frac{1}{\tilde{\lambda}_i}\leq n\frac{1}{\tilde{\lambda}_1}\leq n(n-1)F^{k\bar k}$.
		\end{proof}
	
Sz\'ekelyhidi \cite{Szekelyhidi18} used the $\mathcal{C}$-subsolution to prove the following properties which is a
refinement of Guan \cite[Theorem 2.18]{Guan14} and  the properties are improved in \cite{CM,YR}.   For the equation \eqref{n-1 psh}, it can be proved directly.  	Using the argument in \cite{STW},  we prove
	\begin{lemma}\label{properties coefficients}
		There exist $\kappa$ depending only on $\chi$, n and $h$, such that 
		\begin{enumerate}
			\item we either have
		\begin{equation*}
			f_{k}(\lambda)(\chi_{k\bar{k}}-\lambda_{k})\geq \kappa \sum_{k}f_{k}(\lambda)
		\end{equation*}
	\item  or	$f_{k}\geq \kappa\sum_{i}f_{i}$.
			\end{enumerate}
			In addition, $\sum_{k}f_{k}(\lambda)>\kappa$.
			
	\end{lemma}

	\begin{proof}
		By directly calculation,
		\begin{equation*}
			\sum_{k}f_{k}\chi_{k\bar k}=\sum_{i}\frac{1}{\tilde{\lambda}_i}\tilde{\chi}_{i\bar i}>\tau\sum_{i}\frac{1}{\tilde{\lambda}_i}=\tau\sum_{k}f_{k}(\lambda),
			\end{equation*}
			for some $\tau>0$ depending on a lower bound for $\tilde{\chi}$. We also have
			\begin{equation*}
				\sum_{k}\lambda_{k}f_{k}(\lambda)=n.
				\end{equation*}
Note that if $\lambda_1$ is large enough, then $\tilde{\lambda}_n$ is also large which implies that $\tilde{\lambda}_1$ is small. Then 
			we have the alternative (1) in Proposition \eqref{properties coefficients}. If for any  $\lambda_i$ is bounded, then (2) is true.  In addition,
			\begin{equation*}
				\sum_{k}f_{k}(\lambda)=\sum_{i=1}^n\frac{1}{\tilde{\lambda_i}}\geq n(\tilde{\lambda}_1\cdots\tilde{\lambda}_n)^{-\frac{1}{n}}=ne^{-\frac{(n-1)(h+b)}{n}},
			\end{equation*}
			so the final claim in Proposition \ref{properties coefficients} also holds. 
		\end{proof}
	
Assume that locally $Z=2\re(Z^k_{i\bar{j}}\vp_{k})$ and $W=2\re(W^k_{i\bar{j}}\vp_{k})$.  We need the following properties about $Z$. It was proved in \cite[p. 208]{STW}.
	\begin{lemma}\label{prop Z}
Choose  normal coordinate system such that  $\alpha_{i\bar j}=\delta_{ij}$  at $x_0$. Then at $x_0$, we have 
\begin{enumerate}
	\item $Z_{i\bar j}$  is independent of $\vp_i$ and $\vp_{\bar i}$, i.e., $Z_{i\bar i}^i(x_0)=\overline{Z_{i\bar i}^i}(x_0)=0$.
	\item $\nabla_{i}Z_{i\bar i}(x_0)$ is independent of $\vp_i$.
	\end{enumerate}
		\end{lemma}
		
On the other hand, the linear operator of the equation \eqref{n-1 CY equation} is 
	\begin{equation}
		\begin{split}
		L(\psi)=&F^{i\bar{j}}(\psi_{i\bar j}+2\re(W^k_{i\bar j}\psi_k)+B_{i\bar j}\psi)\\
		=&F^{i\bar{j}}(\psi_{i\bar j}+2\re(W^k_{i\bar j}\psi_k))+\tilde{F}^{i\bar j}\tilde{B}_{i\bar j}\psi.
		\end{split}
		\end{equation}

	\subsection{The astheno-Ricci curvature}
	In this subsection, we will calculate the astheno-Ricci curvature $*\ddbar(\omega^{n-2}).$
Recall the curvature tensor are defined by
\begin{equation*}
	R_{i\bar jk\bar l}=-\partial_{k}\partial_{\bar l}\alpha_{i\bar j}+\alpha^{p\bar q}\partial_{\bar l} \alpha_{p\bar j}\partial _{k}\alpha_{i\bar q}.
\end{equation*}
Define the following Ricci curvature and scalar curvature
\begin{equation*}
	\text{Ric}_{i\bar j}=\alpha^{k\bar l}R_{k\bar l i\bar j},  \text{Ric}^{(2)}_{i\bar j}=\alpha^{k\bar l}R_{ i\bar j k\bar l },
	\text{Ric}^{(3)}_{i\bar j}=\alpha^{k\bar l}R_{ k\bar j i\bar l },
	\text{Ric}^{(4)}_{i\bar j}=\alpha^{k\bar l}R_{ i\bar l k\bar j }
	\end{equation*}
	and 
	\begin{equation*}
		R=\alpha^{i\bar j}	\text{Ric}_{i\bar j}, 	R^{(2)}=\alpha^{i\bar j}	\text{Ric}^{(2)}_{i\bar j}, 	R^{(3)}=\alpha^{i\bar j}	\text{Ric}^{(3)}_{i\bar j}, R^{(4)}=\alpha^{i\bar j}	\text{Ric}^{(4)}_{i\bar j}.
		\end{equation*}
				Denote $T=\partial \alpha$
		and assume  T=$\frac{1}{2}T_{sj\bar k}dz_s\wedge dz_j\wedge d\bar{z}_k$ in local holomorphic coordinates system. Then we define $T_s=\alpha^{j\bar k}T_{sj\bar k}$ and $\tau=T_{s}dz_s$.
		
		 Phong-Picard-Zhang \cite{PPZ19} calculated the $*(\ddbar \alpha^{n-2})$ for conformal balanced metric.
		Now  we give a formula for $*(\ddbar \alpha^{n-2})$ for general Hermitian metric. From the following results, we can see that $*(\ddbar \alpha^{n-2})$ is closed related to the Ricci curvature.
		\begin{proposition}\label{calculate star}
	\begin{enumerate}
		\item If n=3,	\begin{equation*}
			\begin{split}
				&*(\ddbar(\alpha^{n-2}))_{m\bar l}\\[2mm]
				=& \sqrt{-1}(n-2)!(\text{Ric}^{(2)}_{m\bar l}-\text{Ric}^{(3)}_{m\bar l}+\text{Ric}_{m\bar l}-\text{Ric}^{(4)}_{m\bar l}\\[2mm]
				&-\alpha^{j\bar k}\alpha^{s\bar r}T_{mj\bar r}\bar{T}_{\bar l\bar k s})
				+\sqrt{-1}\frac{(n-2)!}{2}(2R-2R^{(3)}-|T|^2)\alpha_{m\bar l}.
			\end{split}
		\end{equation*}
		\item If $n\geq 4$,
		\begin{equation*}
			\begin{split}
				&*(\ddbar(\alpha^{n-2}))_{m\bar l}\\[2mm]
				=& \sqrt{-1}(n-2)!(\text{Ric}^{(2)}_{m\bar l}-\text{Ric}^{(3)}_{m\bar l}+\text{Ric}_{m\bar l}-\text{Ric}^{(4)}_{m\bar l}\\[2mm]
				&-\alpha^{j\bar k}\alpha^{s\bar r}T_{mj\bar r}\bar{T}_{\bar l\bar k s})
				+\sqrt{-1}\frac{(n-2)!}{2}(2R-2R^{(3)}-|T|^2)\alpha_{m\bar l}\\[2mm]
				&+\frac{(n-2)!}{2}\alpha^{p\bar q}\alpha^{s\bar r} (T\wedge \bar{T})_{s\bar rq\bar pm\bar l}
				+\sqrt{-1}\frac{(n-2)!}{6}(3|T|^2-2|\tau|^2)\alpha_{m\bar l}.
			\end{split}
		\end{equation*}
	\end{enumerate}
\end{proposition}
		\medspace
\begin{proof}
We have the following formula in \cite[Lemma 2]{PPZ19}
	\begin{equation}\label{star operator}
		\begin{split}
			(*(\Phi\wedge\alpha^{n-3}))_{k\bar l}=&\sqrt{-1}(n-3)!\alpha^{s\bar r}\Phi_{s\bar r k\bar l}\\[2mm]
			&+\sqrt{-1}\frac{(n-3)!}{2}(\text{tr}\Phi)\alpha_{k\bar l}, n\geq 3;\\[2mm]
			(*(\Psi\wedge\alpha^{n-4}))_{k\bar l}=&\frac{\sqrt{-1}(n-4)!}{2}\alpha^{s\bar r}\alpha^{p\bar q}\Phi_{s\bar r p\bar q k \bar l}\\[2mm]
			&+\sqrt{-1}\frac{(n-4)!}{6}(\text{tr}\Phi)\alpha_{k\bar l}, n\geq 4,
		\end{split}
	\end{equation}
	where $\Phi$ is a $(2,2)$ form and $\Psi$ is a $(3,3)$ form.
	Note 
	\begin{equation*}
		\ddbar \alpha^{n-2}=(n-2)\ddbar \alpha\wedge \alpha^{n-3}+(n-2)(n-3)\sqrt{-1}\partial \alpha\wedge \bar{\partial }\alpha\wedge \alpha^{n-4}.\\[1mm]
	\end{equation*}
	
	Now we calculate $\alpha^{j\bar k}(\ddbar \alpha)_{j\bar km\bar l }, \text{tr}(\ddbar \alpha), 	\alpha^{p\bar q}\alpha^{s\bar r} (T\wedge \bar{T})_{s\bar rq\bar pj\bar k}$ and $	\text{tr} T \wedge T$(see \cite{PPZ19}):
	\begin{equation*}
		\begin{split}
		\alpha^{j\bar k}(\ddbar \alpha)_{j\bar km\bar l }=&\text{Ric}^{(2)}_{m\bar l}-\text{Ric}^{(3)}_{m\bar l}+\text{Ric}_{m\bar l}-\text{Ric}^{(4)}_{m\bar l}
		-\alpha^{j\bar k}\alpha^{s\bar r}T_{mj\bar r}\bar{T}_{\bar l\bar k s},\\[2mm]
	\alpha^{p\bar q}\alpha^{s\bar r} (T\wedge \bar{T})_{s\bar rq\bar pj\bar k}=&\alpha^{p\bar q}\alpha^{s\bar r}(2T_{sj\bar p}\bar{T}_{\bar r\bar k q}+T_{qs \bar k}T_{\bar p\bar r j})-2\alpha^{s\bar r}(T_{js \bar k}T_{\bar r}+T_s\bar{T}_{ \bar k\bar r j})\\[2mm]
	&-2T_{j}\bar T_{\bar k}, 
	\end{split}
\end{equation*}
and 
	\begin{equation*}
		\text{tr}(\ddbar \alpha)=2R-2R^{(3)}-|T|^2, \,\,\,\,	\text{tr} T \wedge T=(3|T|^2-2|\tau|^2).
	\end{equation*}
Using \eqref{star operator}, we have
		\begin{equation*}
			\begin{split}
			(*(\ddbar\alpha\wedge\alpha^{n-2}))_{m\bar l}=&\sqrt{-1}(n-3)!(\text{Ric}^{(2)}_{m\bar l}-\text{Ric}^{(3)}_{m\bar l}+\text{Ric}_{m\bar l}-\text{Ric}^{(4)}_{m\bar l}\\[2mm]
			&-\alpha^{j\bar k}\alpha^{s\bar r}T_{mj\bar r}\bar{T}_{\bar l\bar k s})
			+\sqrt{-1}\frac{(n-3)!}{2}(2R-2R^{(3)}\\[2mm]
			&-|T|^2)\alpha_{m\bar l}.
			\end{split}
			\end{equation*}
			and
			\begin{equation*}
				\begin{split}
			(*\sqrt{-1}\partial \alpha\wedge \bar{\partial }\alpha\wedge \alpha^{n-4})_{m\bar l}=&\frac{(n-4)!}{2}\alpha^{p\bar q}\alpha^{s\bar r} (T\wedge \bar{T})_{s\bar rq\bar pm\bar l}\\[2mm]
			&+\sqrt{-1}\frac{(n-4)!}{6}(3|T|^2-2|\tau|^2)\alpha_{m\bar l}.
			\end{split}
			\end{equation*}
Then if n=3, 
	\begin{equation*}
				\begin{split}
				&(*(\ddbar(\alpha^{n-2})))_{m\bar l}\\[2mm]
				=& \sqrt{-1}(n-2)!(\text{Ric}^{(2)}_{m\bar l}-\text{Ric}^{(3)}_{m\bar l}+\text{Ric}_{m\bar l}-\text{Ric}^{(4)}_{m\bar l}\\[2mm]
				&-\alpha^{j\bar k}\alpha^{s\bar r}T_{mj\bar r}\bar{T}_{\bar l\bar k s})
				+\sqrt{-1}\frac{(n-2)!}{2}(2R-2R^{(3)}-|T|^2)\alpha_{m\bar l}\\[2mm]
				\end{split}
				\end{equation*}
				and if
				$n=4$
							\begin{equation*}
					\begin{split}
						&(*(\ddbar(\alpha^{n-2})))_{m\bar l}\\[2mm]
						=& \sqrt{-1}(n-2)!(\text{Ric}^{(2)}_{m\bar l}-\text{Ric}^{(3)}_{m\bar l}+\text{Ric}_{m\bar l}-\text{Ric}^{(4)}_{m\bar l}\\[2mm]
						&-\alpha^{j\bar k}\alpha^{s\bar r}T_{mj\bar r}\bar{T}_{\bar l\bar k s})
						+\sqrt{-1}\frac{(n-2)!}{2}(2R-2R^{(3)}-|T|^2)\alpha_{m\bar l}\\[2mm]
						&+\frac{(n-2)!}{2}\alpha^{p\bar q}\alpha^{s\bar r} (T\wedge \bar{T})_{s\bar rq\bar pm\bar l}
						+\sqrt{-1}\frac{(n-2)!}{6}(3|T|^2-2|\tau|^2)\alpha_{m\bar l}.
					\end{split}
				\end{equation*}

	\end{proof}

	\section{The $C^{0}$ estimates}
	In this section, we will prove the $C^{0}$ estimates.

			\subsection{The upper bound for $\vp$}

	In this subsection, we  will prove the supremum estimates. Fu-Yau \cite{FY} proved there exists a large positive solution for Fu-Yau equation in complex surface.  In  higher dimensions, it was proved in \cite{CHZ19, CHZ192,PPZ17, PPZ21}. For the form type Calabi-yau equation, we prove that the solution are "nearly non-positive" if the astheno-Ricci curvature is non-positive.
		\begin{proposition}\label{upper bound u}
		Suppose that \begin{equation*}
			\|e^{\vp}\|_{L^{1}(M, \alpha)}=A
		\end{equation*} 
		and $$\sup_{M} \vp\leq 1,$$ 
		then we have
		\begin{equation}
			\sup_{M} \vp\leq M_0A^{\frac{1}{n+1}}.
		\end{equation}
		where $M_0$ is  a uniform constant depending on $\alpha,\omega_0$ and $n$.
	\end{proposition}
	
	\begin{proof}
		Set $\vp_{+}=\max\{\vp, 0\}$. It is easy to see that 
		\begin{equation}\label{positive}
			\int_{M}\vp_{+}\alpha^n\leq  \int_{M}e^{\vp}\alpha^n=A.
		\end{equation}
	By the elliptic condition, we have
		\begin{equation*}
			\begin{split}
				&\int_{M}\vp_{+}^{k}\Big(\omega_0^{n-1} +\ddbar (\vp\alpha^{n-2})\Big)\wedge \alpha\geq 0.
			\end{split}
		\end{equation*}
		Using integration by parts and the Cauchy-Schwarz inequality, 
		\begin{equation*}
			\begin{split}
				&\int_{M}\vp_{+}^k\ddbar(\vp \alpha^{n-2})\wedge \alpha\\
				=&-\frac{4k}{n(k+1)^2}\int_{M}|\partial \vp_{+}^{\frac{k+1}{2}}|^2_{\alpha}\,\alpha^n-k\int_{M}\vp_{+}^{k}\sqrt{-1}\partial (\vp)\wedge\bar{\partial}(\alpha^{n-2})\wedge  \alpha\\
				& -\int_{M}\vp_{+}^k\sqrt{-1}\partial \alpha \wedge \bar{\partial}\vp\wedge \alpha^{n-2}-\int_{M}\vp_{+}^{k+1}\, \sqrt{-1 }\partial \alpha\wedge \bar{\partial} (\alpha^{n-2})\\
				\leq& -\frac{2k}{n(k+1)^2}\int_{M}|\partial \vp_{+}^{\frac{k+1}{2}}|^2_{\alpha}\alpha^n+Ck\int_{M}\vp_{+}^{k+1}\,\alpha^{n}\\
				\leq& -\frac{2k}{n(k+1)^2}\int_{M}|\partial \vp_{+}^{\frac{k+1}{2}}|^2_{\alpha}\,\alpha^n+Ck\int_{M}\vp_{+}^{k}\,\alpha^{n},
			\end{split}
		\end{equation*}
		where in the last inequality we used $\vp\leq 1.$
		Then
		\begin{equation*}
			\frac{2k}{n(k+1)^2}\int_{M}|\partial \vp_{+}^{\frac{k+1}{2}}|^2_{\alpha}\alpha^n\leq Ck\int_{M}\vp_{+}^{k}\alpha^{n}.
		\end{equation*}
		
		It follows, by Sobolev inequality, 
		\begin{equation*}
			\|\vp_{+}^{\frac{(k+1)}{2}}\|_{L^{\frac{2n}{n-1}}}\leq \|\partial \vp_{+}^{\frac{k+1}{2}}\|_{L^{2}}\leq (C(k+1)^2)^{\frac{1}{2}}(\int_{M}\vp_{+}^{k}\alpha^n)^{\frac{1}{2}}.
		\end{equation*}
		Therefore,
		\begin{equation*}
			\|\vp_{+}\|_{L^{\frac{n(k+1)}{n-1}}}\leq (C(k+1)^2)^{\frac{1}{k+1}}\|\vp_{+}\|_{L^{k}}^{\frac{k}{k+1}}, k\geq 1.
		\end{equation*}
	Choose $k_{1}=1, k_{i}=\frac{n}{n-1}(k_{i-1}+1)$. It is easy to see that 
		\begin{equation}\label{ki}
			k_{i}=(\frac{n}{n-1})^{i-1}(k_1+n)-n.
		\end{equation}
		Then we have
		\begin{equation}\label{iteration}
			\begin{split}
				&\|\vp_{+}\|_{L^{k_{i+1}}}\leq ((C(k_i+1))^{\frac{1}{k_i+1}}\|\vp_{+}\|_{L^{k_i}}^{\frac{k_i}{k_i+1}}\\
				\leq & ((C(k_i+1))^{\frac{1}{k_i+1}}((C(k_{i-1}+1))^{\frac{k_i}{k_i+1}\frac{1}{k_{i-1}+1}}\|\vp_{+}\|_{L^{k_{i-2}}}^{\frac{k_i}{k_i+1}\frac{k_{i-1}}{k_{i-1}+1}}\\
				\leq & ((C(k_i+1))^{\frac{1}{k_i+1}}((C(k_{i-1}+1))^{\frac{k_i}{k_i+1}\frac{1}{k_{i-1}+1}}\cdots (C(k_1+1))^{\frac{k_i}{k_i+1}\frac{k_{i-1}}{k_{i-1}+1}\cdots \frac{1}{k_1+1}}\\
				&\|\vp_{+}\|_{L^{k_{1}}}^{\frac{k_i}{k_i+1}\frac{k_{i-1}}{k_{i-1}+1}\cdots \frac{k_1}{k_1+1}}\\
				=& C^{\frac{1}{k_{i+1}}(\frac{n}{n-1}+(\frac{n}{n-1})^2+\cdots+(\frac{n}{n-1})^i )}(k_i+1)^{\frac{1}{k_i+1}}(k_{i-1}+1)^{\frac{1}{k_i+1}\frac{n}{n-1}}\cdots \\ & (k_1+1)^{\frac{1}{k_i+1}(\frac{n}{n-1})^{i-1}}
				\|\vp_{+}\|_{L^{1}}^{\frac{k_i}{k_i+1}\frac{k_{i-1}}{k_{i-1}+1}\cdots \frac{k_1}{k_1+1}}.
			\end{split}   
		\end{equation}
		Using $k_{i}=\frac{n}{n-1}(k_{i-1}+1)$ and \eqref{ki} , we have
		\begin{equation}
			\frac{k_i}{k_i+1}\frac{k_{i-1}}{k_{i-1}+1}\cdots \frac{k_1}{k_1+1}=\frac{k_1}{(\frac{n-1}{n})^{i-1}k_{i+1}}\rightarrow \frac{k_1}{k_1+n}, \,\,\text{when\,\,} i\rightarrow \infty.
		\end{equation}
		In addition, we have
		
		\begin{equation}
			\begin{split}
				&C^{\frac{1}{k_i+1}+\frac{k_i}{k_i+1}\frac{1}{k_{i-1}+1}+\cdots+\frac{k_i}{k_i+1}\frac{k_{i-1}}{k_{i-1}+1}\cdots \frac{1}{k_1+1}}\\
				=&C^{\frac{1}{k_{i+1}}(\frac{n}{n-1}+(\frac{n}{n-1})^2+\cdots+(\frac{n}{n-1})^i )}\\
				=&C^{\frac{n((\frac{n}{n-1})^{i}-1)}{(\frac{n}{n-1})^{i}(k_1+n)-n}}\rightarrow
				C^{\frac{n}{(k_1+n)}}, \text{\,when\,} i\rightarrow \infty.
			\end{split}
		\end{equation}
		Using \eqref{ki}, we have 
		\begin{equation}
			\begin{split}
				&(k_i+1)^{\frac{1}{k_i+1}}(k_{i-1}+1)^{\frac{k_i}{k_i+1}\frac{1}{k_{i-1}+1}}\cdots (k_1+1)^{\frac{k_i}{k_i+1}\frac{k_{i-1}}{k_{i-1}+1}\cdots \frac{1}{k_1+1}}\\
				=&  (k_i+1)^{\frac{1}{k_i+1}}(k_{i-1}+1)^{\frac{1}{k_i+1}\frac{n}{n-1}}\cdots (k_1+1)^{\frac{1}{k_i+1}(\frac{n}{n-1})^{i-1}}\\
				=&e^{\sum_{a=0}^{i-1}\frac{(\frac{n}{n-1})^{a}\ln(k_{i-a}+1)}{k_i+1}}<C, \text{\,\,when\,\,} i\rightarrow \infty.
			\end{split}
		\end{equation}
		By \eqref{positive},	letting $i\rightarrow \infty$ in \eqref{iteration}, we have
		\begin{equation}
			\|\vp_{+}\|_{L^{\infty}}\leq C\|\vp_{+}\|_{L^{1}}^{\frac{1}{n+1}}\leq CA^{\frac{1}{n+1}}.
		\end{equation}
		
	\end{proof}

	\subsection{The lower bound for $\vp$}
	
		First  we use  Sz\'ekelyhidi 's argument \cite{Szekelyhidi18}, to prove the $L^{1}$ estimates. Consider the operator
	\begin{equation}\label{Delta}
		\Delta \psi:=\Delta_{\alpha} \psi+\text{tr}_{\alpha}W(\psi)+ (\text{tr}_{\alpha}B) \vp,
	\end{equation}
	where $\Delta_{\alpha}\psi=\alpha^{i\bar j}\psi_{i\bar j}$ is the Chern Laplace operator, $\text{tr}_{\alpha}W(\psi), \text{tr}_{\alpha}B$ are the trace of $W(\psi)$, $B$, respectively,  and $(\alpha^{i\bar j})$ is the inverse matrix of $(\alpha_{i\bar j})$. 

	\begin{proposition}\label{L1}
		Let $(M,\alpha)$ be a compact Hermitian manifold with $\int_{M}\alpha^n=1$.  Suppose that $\vp$ satisfies
		\begin{equation}\label{norm}
			\int_{M}e^{\vp}\,\alpha^n=A, \quad \Delta \vp \geq -C_0
		\end{equation}
		and $\sup\vp\leq 1$
		for some constant $C_0$. Then there exists a constant $C$ depending on $ A, C_0,\omega_0,h$ and $(M, \alpha)$ such that
		\begin{equation*}
			\int_{M}|\vp|\,\alpha^{n}\leq C.
		\end{equation*}
	\end{proposition}
	
	\begin{proof}
		By \eqref{norm}, $\sup_{M}e^{\vp}\geq A$, which implies 
		\begin{equation}\label{sup inf}
		\sup\vp\geq \ln A
		\end{equation}
		Using Proposition \ref{upper bound u}, $|\sup_{M}\vp|\leq C$.
	Note $\alpha^{i\bar j}B_{i\bar j}\leq0$. 
	By the argument in \cite{Szekelyhidi18}, using weakly Harnack inequality, we can prove  a bound on the integral  of
	$|\sup\vp-\vp|$.  This implies Proposition \ref{L1}.

	\end{proof}

In the following, the infimum estimates will be proved. We follow the arguments of \cite{Szekelyhidi18}, which are generalizations of the arguments in \cite{Blocki05,Blocki11}. In \cite{GM}, George also proved
a similar result. 
	\begin{proposition}\label{Prop32}
		 Let $\vp$ be a smooth solution of \eqref{n-1 CY equation} and $\sup\vp\leq 1$. Suppose \[*(\ddbar \alpha)^{n-2}\leq0.\]
		 Then there exists a constant $C$ depending on $A, \|h\|_{C^{0}}$, $\|\omega_0\|_{C^{0}}$ and $(M,\alpha)$ such that
		\begin{equation}\label{inf} 
			\inf_{M} \vp\geq -C.
		\end{equation}
	\end{proposition}

	\begin{proof}
		If $\inf_{M}\vp \geq 0$,  the proof is complete. Now we assume $\inf_{M} \vp<0 $
and	denote $I=\inf_{M}\vp$. Assume $I=\vp(x_0)$ and choose a local coordinate chart $(z_{1},\ldots,z_{n})$ near $x_{0}$ containing the unit ball $B_{1}(0)\subset \mathbb{R}^{2n}$ such that the point $x_{0}$ corresponds to the origin $0\in \mathbb{R}^{2n}$.
		Define $v= \vp+\ve\sum_{i=1}^{2n}|z_{i}|^{2}$	for a small $\ve>0$ . Then
		\[
		v(0) = I, \quad
		v\geq I+\ve \ \text{on $\partial B_{1}(0)$}.
		\]
Denote
		\begin{equation}\label{P}
			\begin{split}
			P=\Big\{x\in B_{1}(0) : |Dv(x)|\leq &\frac{\ve}{2},
			v(y)\geq v(x)+Dv(x)\cdot (y-x),\\
			&\ \forall\  y\in B_{1}(0)\Big\}.
			\end{split}
		\end{equation}
	Choose $\epsilon$ small enough such that 
$
		I +\epsilon\leq 0.
$
	Then for any $p\in P$,
	\begin{equation}\label{negative p}
		 v(p)\leq v(0)+|Dv| |p|\leq I+\epsilon\leq 0.
		\end{equation}
		Using  the modified Alexandroff-
		Bakelman-Pucci maximum principle \cite[Proposition 10]{Szekelyhidi18}, we have
		\begin{equation}\label{3.5}
			c_{0}\ve^{2n}\leq \int_{P}\det(D^{2}v),
		\end{equation}
		where $c_{0}$ is a constant depending only on $n$.
		
		\begin{claim}\label{bound det}
		On $P$,		$\det(D^{2}v) \leq C$.
			\end{claim}
\begin{proof}[Proof of Claim \ref{bound det}]
		Note that $v$ is a convex function on $P$ and so $D^{2}v\geq 0$ on $P$. Then
	\begin{equation}\label{zero estimate D2}
		\ddbar \vp \geq \ddbar v-C\ve\,\mathrm{Id}\geq -C\ve\,\mathrm{Id} \ \ \text{on $P$}.
	\end{equation}
	Combing with \eqref{negative p}, 	$\vp\tilde{B}\geq 0$ on $P$. If we choose $\epsilon$ small enough, then
	\[
	\tilde{\chi}+\frac{1}{n-1}((\Delta \vp)\alpha-\ddbar \vp)+Z+\vp \tilde{B}
	\geq \frac{1}{2}\tilde{\chi},\]
	which implies $\tilde{\lambda}_{i}\geq \frac{1}{C}$. By \eqref{n-1 ma}, we have
	\[
	\tilde{\chi}+\frac{1}{n-1}((\Delta \vp)\alpha-\ddbar \vp)+Z+\vp \tilde{B}
	\leq C\tilde{\chi},\]
	and so	$	|\ddbar \vp|\leq C \text{\,\,on\,\,} P.$
Then 
		\begin{equation}\label{c0 upper c2}
	|\ddbar v|\leq C  \text{\,\,on\,\,} P.
		\end{equation}
Recall that  for Hermitian matrices $A\geq 0$ and $N\geq 0$, $\det(A+N)\geq \det(A)+\det(N)$. On $P$, we have $D^{2}v\geq 0$ and so
	\[
	\det(\ddbar v) = 2^{-2n}\det(D^{2}v+J^{T}\cdot D^{2}v\cdot J)
	\geq 2^{-2n+1}\det(D^{2}v).
	\]
	Combining this with \eqref{c0 upper c2},
	\[
	\det(D^{2}v) \leq 2^{2n-1}\det(\ddbar v)\leq C.
	\]
	\end{proof}			

		Using Claim \ref{bound det} and (\ref{3.5}), 
		\begin{equation}\label{zero estimates 3}
			c_{0}\ve^{2n}\leq C|P|.
		\end{equation}
		For each $x\in P$, choosing $y=0$ in $\eqref{P}$, then
		Since $|Dv|\leq \frac{\epsilon}{2}$, 
		\begin{equation*}
			I=v(0)\geq v(x)-\epsilon.
		\end{equation*}
		Since $I+\ve\leq 0$, it follows that
	$
		|I+\ve| \leq -v \ \ \text{on $P$}.
$
 Using \eqref{zero estimates 3}, it follows
		\[
		c_{0}\ve^{2n} \leq |P| \leq \frac{\int_{P}(-v)\chi^{n}}{|I+\ve|}.
		\]
$
	|I+\ve| \leq \frac{1}{|P|}\int_{P}(-v)\chi^{n} \leq C.
	$
		This gives the required estimate of $I$.
	\end{proof}

	\section{Second order estimates}\label{second order estimate}
		In this section, we prove the second order estimates. 
	\begin{theorem}\label{Thm4.1}
		Suppose that 
		\begin{equation}\label{assum for 2nd}
		-C_0\leq	\vp\leq M_0A^{\frac{1}{n+1}}.
			\end{equation}
	 There exists a constant $A_0$ and $C$ such that if $A\leq A_0$
		\begin{equation}\label{HMW}
			\sup_{M}|\ddbar \vp|_{\alpha} \leq C(\sup_{M}|\partial \vp|_{\alpha}^{2}+1),
		\end{equation}
		where $C$ depends on $(M, \alpha)$, $\omega_0$, $ h,A$ and $A_{0}$ is a constant such that $$M_0A^{\frac{1}{n+1}}_0\|\tilde{B}\|_{\alpha}\leq\frac{\kappa}{4}.$$ 
			\end{theorem}

	In this paper, we compute using covariant derivatives with respect to the Chern connection of $\alpha$. 
		Denote $g_{i\ov{j},k}=\nabla_{k}g_{i\bar j}$ and  $g_{i\ov{j},k\ov{l}}=\nabla_{\bar l}\nabla_{k}g_{i\bar j}$.
		Let us recall the commutation formulas for covariant derivatives:
	\begin{equation}\label{commutation formulas}
		\begin{split}
				\vp_{i\ov{j}k} {} & = \vp_{k\ov{j}i}-T_{ki}^{p}\vp_{p\ov{j}}, \\
		\vp_{i\ov{j}k\ov{l}}{}&=\vp_{k\ov{l}i\ov{j}}-T_{ki}^{p}\vp_{p\ov{l}\ov{j}}  -\ov{T_{lj}^{q}}\vp_{k\ov{q}i}
			+\vp_{p\ov{j}}R_{k\ov{l}i}^{\ \ \ \ p}-\vp_{p\ov{l}}R_{i\ov{j}k}^{\ \ \ p}
			-T_{ik}^{p}\ov{T_{lj}^{q}}\vp_{p\ov{q}},\\
		\end{split}
	\end{equation}
		where  $T_{ij}^{k}$ and $R_{i\ov{j}k}^{\ \ \ \ l}$ are components of torsion tensor and curvature tensor  induced by the Chern connection of $\alpha$.

Suppose that  $\lambda_{1}\geq\lambda_{2}\geq\cdots\geq\lambda_{n}$ are the eigenvalues of $\Phi=(\Phi^{i}_{k})=(\alpha^{i\bar j}\omega_{k\bar j})$ with respect to $\alpha$. For convenience, we write
	$|\cdot| = |\cdot|_{\alpha}$
	and
	\begin{equation*}
		K=\sup_{M}|\partial \vp|^{2}+1.
	\end{equation*}
On $\Omega=\{\lambda_1>0\}\subset M$, consider
	\begin{equation*}
		Q = \log \lambda_1+\eta(|\partial \vp|^{2})+\xi(\vp)
	\end{equation*}
 where
	\[
	\xi(\vp) = D_1e^{-D_2(\vp-\sup_{M}\vp-1)}, \quad
	\eta(s) = -\frac{1}{2}\log(2K-s)
	,\]
for  constants $D_1$ and $D_2$ to be determined later.
	By directly calculation, 
	\begin{equation}\label{xieta}
		\begin{split}
		\eta''=2(\eta')^{2}, & \quad  \frac{1}{8K} \leq \eta' \leq \frac{1}{4K} . 
		\end{split}
	\end{equation}
	We may assume $\Omega\neq \emptyset$, otherwise we are done. Note $Q(z)\ri -\infty$ as $z$ approaches to $\partial{\Omega}$. Let $Q(x_0)=\max_{M}Q(x_0)$, where $x_{0}\in \Omega$.
Using the standard argument \cite{Szekelyhidi18,STW,TW17}, to prove Theorem \ref{Thm4.1}, it suffices to show
	\begin{equation}\label{goal}
		\lambda_1(x_{0}) \leq CK.
	\end{equation}

	Near $x_{0}$, choose a local normal coordinates $(z_1,\cdots,z_n)$ such that at $x_{0}$,
	\begin{equation}\label{5..6}
		\alpha_{i\overline{j}}=\delta_{ij}, \quad g_{i\overline{j}}=\delta_{ij}g_{i\overline{i}}, \quad g_{1\overline{1}}\geq g_{2\overline{2}}\geq\cdots\geq g_{n\overline{n}}.
	\end{equation}
and
	\[
	F^{1\ov{1}} \leq F^{2\ov{2}} \leq \cdots \leq F^{n\ov{n}}.
	\]
 We assume that $\lambda_1$ is smooth and $\lambda_1>\lambda_{2}$ at $x_{0}$,  using a viscosity argument (see \cite{Szekelyhidi18,STW,YR})

	\medskip
	Applying the maximum principle  at $x_{0}$, we see that
	\begin{equation}\label{L Q}
		\begin{split}
			0  \geq L(Q) 
			=& \frac{L(\lambda_1)}{\lambda_1}- F^{i\bar{i}}\frac{|(\lambda_1)_{i}|^{2}}{\lambda_1^{2}}
			+\eta'L(|\partial \vp|^{2})\\
			&+\eta'' F^{i\bar{i}}|(|\de \vp|^{2})_{i}|^{2} 
			+\xi'L(\vp)+\xi''F^{i\bar{i}}|\vp_{i}|^{2}.
		\end{split}
	\end{equation}

	In the sequel, we will use the Einstein summation,
and use $C$ to denote a constant depending on $\|\vp\|_{C^{0}}$, $h$, $\omega_0$,  $(M,\alpha)$, and $C_{D}$ to denote a constant further depending  on $D_1$ and $D_2$,

	\subsubsection{Lower bound for $L(Q)$}
	
	\begin{proposition}\label{lower bound of L Q}
		For $\ve\in (0,\frac{1}{3}]$, at $x_{0}$, we have
		\begin{equation}\label{LQ0}
	\begin{split}
		0\geq \,&
		G-\sum_{k=1}^n\frac{C}{\lambda_1}F^{i\bar i}|g_{1\bar 1, i}|- N
		+\frac{3\eta'}{4} \sum_{j}F^{i\bar{i}}(|\vp_{ij}|^{2}+|\vp_{i\bar j}|^{2})\\
		&+\eta'' F^{i\bar{i}}|(|\de \vp|^{2})_{i}|^{2} +\xi'L(\vp)+\xi'' F^{i\bar{i}}|\vp_{i}|^{2}
		-C\mathcal{F},
	\end{split}    
\end{equation}
		where
		\[
		G= -\frac{1}{\lambda_1} F^{i\bar{j},k\bar{l}}\nabla_{1}g_{i\ov{j}}\nabla_{\ov{1}}g_{k\ov{l}}, N=F^{i\bar{i}}\frac{|(\lambda_1)_{i}|^{2}}{\lambda_1^{2}}.
		\]
	\end{proposition}

	\medskip
The third term is the bad term that we need to control. 
To prove Proposition \ref{lower bound of L Q}, we shall estimate the lower bounds of $L(\lambda_1)$ and $L(|\de u|^{2})$, respectively.
First we deal with the term $F^{i\bar{i}}g_{1\bar 1,i\bar{i}}$ at $x_0$. We have the following inequality.
	\medskip 
		\begin{proposition}\label{claim 1}
	Assume $\lambda_1\geq K$.			At $x_{0}$, we have
	\begin{equation*}
		\begin{split}
			F^{i\bar{i}}g_{1\bar 1,i\bar{i}}
			\geq \,& -F^{i\ov{j},k\ov{l}}\nabla_{1}g_{i\ov{j}}\nabla_{\ov{1}}g_{k\ov{l}}-\sum_{k=1}^nCF^{i\bar i}|g_{1\bar k, i}|-\frac{1}{32K}F^{i\bar i}|\vp_{pi}|^2\\
			&-C\mathcal{F}\lambda_1.
		\end{split}
	\end{equation*}
\end{proposition}

\begin{proof}
	By differentiating \eqref{n-1 psh} along $\nabla_{\ov{1}}\nabla_{1}$, 
	\begin{equation}\label{twice deriv}
		\begin{split}
		F^{i\ov{i}}g_{i\bar i,1\bar 1}&=	F^{i\ov{i}}\nabla_{\ov{1}}\nabla_{1}(\chi_{i\ov{i}}+\vp_{i\bar i}+W_{i\bar i}+\vp B_{i\bar i})\\
		&= (n-1)h_{1\ov{1}}-F^{i\ov{j},k\ov{l}}\nabla_{1}g_{i\ov{j}}\nabla_{\ov{1}}g_{k\ov{l}}.
		\end{split}
	\end{equation}
On the other hand,	by \eqref{commutation formulas},
\begin{equation}\label{comm 1}
	F^{i\bar{i}}\vp_{1\bar 1 i\bar i}= F^{i\bar i}\vp_{i\bar i1\bar 1}+\sum_{k=1}^nF^{i\bar i}O(|\vp_{1\bar k i}|)+O(\lambda_1)\mathcal{F},
\end{equation}
where $O(|\vp_{1\bar k i}|)$, $O(\lambda_1)$ mean the terms can controlled by $|\vp_{1\bar k i}|$, $\lambda_1$  respectively.	
Using  \eqref{comm 1} and \eqref{twice deriv}, it follows
\begin{equation*}
	\begin{split}
		&F^{i\ov{i}}\vp_{1\bar 1i\bar i}+F^{i\bar i}W_{i\bar i,1\bar 1}\\
		\geq &-C-F^{i\ov{j},k\ov{l}}\nabla_{1}g_{i\ov{j}}\nabla_{\ov{1}}g_{k\ov{l}}	-C\mathcal{F}\lambda_1
		-\sum_{k=1}^nCF^{i\bar i}|\vp_{1\bar k i}|.
	\end{split}
\end{equation*}
Therefore,
\begin{equation*}
	\begin{split}
		F^{i\bar i}g_{1\bar1, i\bar i}=&F^{i\bar i}\chi_{1\bar 1,i\bar i}+F^{i\bar i}\vp_{1\bar 1i\bar i}+F^{i\bar i}W_{1\bar1,i\bar i}+F^{i\bar i}(\vp B)_{1\bar 1,i\bar i}\\
		\geq & -F^{i\ov{j},k\ov{l}}\nabla_{1}g_{i\ov{j}}\nabla_{\ov{1}}g_{k\ov{l}}	-\sum_{k=1}^nCF^{i\bar i}|\vp_{1\bar k i}|-C\mathcal{F}\lambda_1\\
		&+F^{i\bar i}W_{1\bar1,i\bar i}-F^{i\bar i}W_{i\bar i,1\bar 1}-C.
	\end{split}
\end{equation*}
Note
$
	\vp_{1\bar k i}=g_{1\bar k,\bar 1}-\chi_{1\bar k, i}-W_{1\bar k, i}-B_{1\bar k,i}.
$
We have
\begin{equation}
	\begin{split}
	\sum_{k=1}^nF^{i\bar i}|\vp_{1\bar k i}|&\leq C(\sum_{p}F^{i\bar i}|\vp_{p}|+\lambda_1\mathcal{F})\\
	&\leq\frac{1}{96K}\sum_{p}F^{i\bar i}|\vp_{pi}|^2+C\lambda_1\mathcal{F}+\sum_{k=1}^nCF^{i\bar i}|g_{1\bar k, i}|.
	\end{split}
	\end{equation}
It follows that
\begin{equation}\label{eigen inequality}
	\begin{split}
		F^{i\bar i}g_{1\bar1, i\bar i}
		\geq & -F^{i\ov{j},k\ov{l}}\nabla_{1}g_{i\ov{j}}\nabla_{\ov{1}}g_{k\ov{l}}	-\sum_{k=1}^nF^{i\bar i}|g_{1\bar k, i}|-C\mathcal{F}\lambda_1\\
		&+F^{i\bar i}W_{1\bar1,i\bar i}-F^{i\bar i}W_{i\bar i,1\bar 1}-\frac{1}{96K}\sum_{p}F^{i\bar i}|\vp_{pi}|^2.
	\end{split}
\end{equation}

Now we deal the terms involving $W$.	First we deal with $F^{i\bar i}W_{i\bar i,1\bar1}$.  
\begin{lemma}\label{W 1}
	If $\lambda_1\geq K$, we have
	\begin{equation*}
		F^{i\bar i}W_{i\bar i, 1\bar 1}\leq \frac{1}{96K}\sum_{p}F^{i\bar i}|\vp_{pi}|^2+C(F^{i\bar i}|g_{1\bar 1,i}|+\lambda_{1}\mathcal{F}).
	\end{equation*}
\end{lemma}

\begin{proof}
	By \eqref{coef copar} and \eqref{trans}, we have
	\begin{equation*}
		\begin{split}
			F^{i\bar i}W_{i\bar i,1\bar 1}&=\frac{1}{n-1}\sum_{k}\sum_{i\neq k}\tilde{F}^{i\bar i}W_{k\bar k, 1\bar 1}\\
			&=\frac{1}{n-1}\sum_{i}\tilde{F}^{i\bar i}\sum_{k\neq i}	W_{k\bar k,1\bar 1}=\tilde{F}^{i\bar i}Z_{i\bar i,1\bar 1}.
		\end{split}	
	\end{equation*}
	By Lemma \ref{prop Z}, $Z_{1\bar 1,1\bar 1}$ does not contain the terms $\vp_{11\bar 1}$ and $\vp_{1\bar 1}$ or their complex conjugates.	By Lemma \ref{prop f}(3) (4), it follows that
	\begin{equation*}
		\begin{split}
			F^{i\bar i}W_{i\bar i,1\bar 1}\leq &C(\tilde{F}^{1\bar 1}\sum_{k>1}(|\vp_{k1}|+|\vp_{k\bar 1 1}|)+\sum_{i>1,k}\tilde{F}^{i\bar i}(|\vp_{k1}|+|\vp_{k\bar 1 1}|)+\lambda_1\mathcal{F})\\
			\leq &C(\sum_{i>1}F^{i\bar i}(|\vp_{i1}|+|\vp_{k\bar 1 1}|)+F^{1\bar 1}\sum_{k}(|\vp_{k1}|+|\vp_{k\bar 1 1})+\lambda_1\mathcal{F})\\
			\leq & C(F^{1\bar 1}(|\vp_{11}|+|\vp_{1\bar 1 1}|)+\sum_{i>1}F^{i\bar i}(|\vp_{i1}|+|\vp_{i\bar 1 1}|)+\lambda_1\mathcal{F})\\
			\leq & C(\sum_{p}F^{i\bar i}|\vp_{pi}|+F^{i\bar i}|\vp_{1\bar 1 i}|+\lambda_1\mathcal{F}).
		\end{split}
	\end{equation*}
	Here in the third inequality, we used $F^{1\bar 1}\leq F^{i\bar i}$.
	Note that 
	\begin{equation*}
		\vp_{1\bar 1 i}=g_{1\bar 1,\bar i}-\chi_{1\bar 1, i}-W_{1\bar 1, i}-(\vp B_{1\bar 1})_{i},
	\end{equation*}
	which implies
	\begin{equation*}
		\begin{split}
		F^{i\bar i}|\vp_{1\bar 1 k}|\leq & F^{i\bar i}|g_{1\bar 1, i}|+C(\sum_{p}F^{i\bar i}|\vp_{pi}|+\lambda_1\mathcal{F}).
		\end{split}
	\end{equation*}
Using  the Cauchy-Schwarz  inequality and $\lambda_1\geq K$, we have
	\begin{equation*}
		F^{i\bar i}W_{i\bar i, 1\bar 1}\leq \frac{1}{96K}\sum_{p}F^{i\bar i}|\vp_{pi}|^2+C(F^{i\bar i}|g_{1\bar 1,i}|+\lambda_{1}\mathcal{F}).
	\end{equation*}
\end{proof}

Next we deal with the $F^{i\bar i}W_{1\bar 1, i\bar i}$. We have the following inequality.
\begin{lemma}\label{W 2}
	At $x_0$, we have
	\begin{equation}
		F^{i\bar i}W_{1\bar 1,i\bar i} \geq -\sum_{p}\frac{1}{96K}F^{i\bar i}|\vp_{pi}|^2-C\lambda_1\mathcal{F}.
	\end{equation}
\end{lemma}

\begin{proof}
	By directly calculation,
	\begin{equation}\label{W term}
		\begin{split}
			&F^{i\bar i}W_{1\bar 1,i\bar i}=F^{i\bar i}2\re(\nabla_{\bar i}\nabla_{i}(W_{1\bar 1}^p\vp_{p}))\\
			=&2\re(F^{i\bar i}W^p_{1\bar 1}\vp_{pi\bar i})+O(F^{i\bar i}|\vp_{pi}|)+O(\lambda_1)\mathcal{F}.
		\end{split}
	\end{equation}
	Differentiating \eqref{n-1 psh} along $\nabla_{p}$, we have
	\begin{equation*}
		F^{i\bar i}g_{i\bar i, p}=(n-1)h_{p},
	\end{equation*}
	which implies
	\begin{equation}\label{third order 1}
		\begin{split}
			|F^{i\bar i}\vp_{i\bar ip}|=&|F^{i\bar i}(g_{i\bar i, p}-\chi_{i\bar i, p}-W_{i\bar{i},p}-(\vp B_{i\bar i})_p)|\\
			\leq &CK^{\frac{1}{2}}\mathcal{F}+|F^{i\bar i}W_{i\bar i,p}|.
		\end{split}
	\end{equation}
	Here we used $K\geq1$. Now we deal with $|F^{i\bar i}W_{i\bar i,p}|$.
	By \eqref{trans}, we have
	\begin{equation*}
		F^{i\bar i}W_{i\bar i,p}=\tilde{F}^{i\bar i}Z_{i\bar i,p}.
	\end{equation*}
	Then, using Lemma \ref{prop Z} and \eqref{coef copar},  
	\begin{equation*}
		\begin{split}
			F^{i\bar i}W_{i\bar i, p}\leq &C(\tilde{F}^{1\bar 1}\sum_{k>1}(|\vp_{kp}|+|\vp_{\bar kp}|)+\tilde{F}^{i\bar i}(|\vp_{ip}|+|\vp_{\bar ip}|)+K^{\frac{1}{2}}\mathcal{F})\\
			\leq & C(\sum_{k>1}F^{i\bar i}((|\vp_{ip}|+|\vp_{\bar ip}|)+F^{1\bar 1}\sum_{k}(|\vp_{kp}|+|\vp_{\bar kp}|)+K^{\frac{1}{2}}\mathcal{F})\\
			\leq & C(\sum_{p}F^{i\bar i}(|\vp_{pi}|+|\vp_{\bar ip}|)+K^{\frac{1}{2}}\mathcal{F}).
		\end{split}
	\end{equation*}
	Combing this inequality with \eqref{third order 1},
	\begin{equation}\label{third order 2}
		|F^{i\bar i}\vp_{i\bar i p}|\leq \sum_{p}F^{i\bar i}(|\vp_{pi}|+|\vp_{\bar ip}|)+K^{\frac{1}{2}}\mathcal{F}).
	\end{equation}
	From \eqref{W term}, we have
	\begin{equation*}
		\begin{split}
		F^{i\bar i}W_{1\bar 1,i\bar i} \geq &-C(\sum_{p}F^{i\bar i}|\vp_{pi}|+\lambda_1\mathcal{F})\\
		\geq & -\sum_{p}\frac{1}{96K}F^{i\bar i}|\vp_{pi}|^2-C\lambda_1\mathcal{F}.\\
		\end{split}
	\end{equation*}
\end{proof}
Combing \eqref{eigen inequality}, Lemma \ref{W 1} with Lemma \ref{W 2}, 
\begin{equation*}
	\begin{split}
		F^{i\bar{i}}g_{1\bar 1,i\bar{i}}
		\geq & -F^{i\ov{j},k\ov{l}}\nabla_{1}g_{i\ov{j}}\nabla_{\ov{1}}g_{k\ov{l}}-\sum_{k=1}^nCF^{i\bar i}|g_{1\bar k, i}|\\
		&-\frac{1}{32K}F^{i\bar i}|\vp_{pi}|^2-C\lambda_1\mathcal{F}.
	\end{split}
\end{equation*}

\end{proof}

Using the inequality about $L(g_{1\bar 1})$ at $x_0$, we immediately obtain the following inequality for $L(\lambda_1)$.
	\begin{proposition}\label{lower bound of L lambda1}
		Assume $\lambda_1\geq CK$.
		For each $\ve\in(0,\frac{1}{3}]$, at $x_{0}$, we have
		\begin{equation}
			\begin{split}
		L(\lambda_1) \geq &\sum_{p>1}F^{i\bar{i}}\frac{|g_{1\bar p,i}|^2}{8n\lambda_1}
		-F^{i\ov{j},k\ov{l}}\nabla_{1}g_{i\ov{j}}\nabla_{\ov{1}}g_{k\ov{l}}-\sum_{k=1}^nCF^{i\bar i}|g_{1\bar 1, i}|\\
		&-\frac{1}{32K}F^{k\bar k}|\vp_{pk}|^2-C\lambda_1\mathcal{F}.
		\end{split}
		\end{equation}
	\end{proposition}
	
	\begin{proof}
	We need the following formulas (see e.g. \cite{CTW19,Spruck05,Szekelyhidi18}):
		\begin{equation*}
			\begin{split}
				\frac{\partial \lambda_1}{\partial \Phi^{q}_{p} }
				= {} & \delta_{1p}\delta_{1q}, \\
				\frac{\partial^{2} \lambda_1}{\partial \Phi^{q}_{p}\partial \Phi^{s}_{r}}
				= {} & (1-\delta_{1p})\frac{\delta_{1q}\delta_{1r}\delta_{ps}}{\lambda_1-\lambda_p} +(1-\delta_{1r})\frac{\delta_{1s}\delta_{1p}\delta_{rq}}{\lambda_1-\lambda_r}.
			\end{split}
		\end{equation*}
	By directly calculation, 
		\begin{equation}\label{lower bound of L lambda1 eqn 1}
			\begin{split}
				L(\lambda_1)
				= {} & F^{i\bar{i}}\frac{\partial^{2} \lambda_1}{\partial \Phi^{q}_{p}\partial\Phi^{s}_{r}}\nbi(\Phi^{s}_{r})\nbbi(\Phi^{q}_{p})
				+F^{i\bar{i}}\frac{\partial \lambda_1}{\partial\Phi^{q}_{p}}\nbi\nbbi
				(\Phi^{q}_{p})\\
				= {} &F^{i\bar{i}}g_{1\bar 1,i\bar{i}}+\sum_{p>1} \frac{F^{i\bar{i}}(|g_{1\bar p,i}|^2+|g_{p\bar 1,i}|^2)}{\lambda_1-\lambda_p}\\
				\geq {} &F^{i\bar{i}}g_{1\bar 1,i\bar{i}}+ \sum_{p>1}F^{i\bar{i}}\frac{|g_{1\bar p,i}|^2}{4n\lambda_1}.
			\end{split}
		\end{equation}
Using the Cauchy-Schwarz inequality, we have for $k\neq 1$,
		\begin{equation}
			|g_{1\bar k, i}|\leq \frac{1}{8n\lambda_1}	|g_{1\bar k,i }|^2+C\lambda_1.
		\end{equation}
		Combining this inequality with  \eqref{lower bound of L lambda1 eqn 1} and Proposition \ref{claim 1}, we obtain Lemma \ref{lower bound of L lambda1}.
		\end{proof}

	\medskip
	Now, we give the lower bound of $L(|\de \vp|^{2})$.
	
	\begin{lemma}\label{lower bound of L de u}
		At $x_{0}$, we have
		\begin{equation}\label{3.7}
			L(|\partial \vp|^{2}) \geq \frac{3}{4} \sum_{j}F^{i\bar{i}}(|\vp_{ij}|^{2}+|\vp_{i\bar j}^{2})-CK\mathcal{F}.
		\end{equation}
	\end{lemma}
	
	\begin{proof}
By \eqref{commutation formulas}, \eqref{third order 2} and the Cauchy-Schwarz inequality, we have
		\begin{equation}\label{Lemma 1 equation 1}
			\begin{split}
				& F^{i\ov{i}}\big((|\de\vp|_{g}^{2})_{i\ov{i}} \\
				= {} & \sum_{k}F^{i\ov{i}}\left(|\vp_{ki}|^{2}+|\vp_{k\bar i}|^{2}\right)
				+2\textrm{Re}(\sum_{k}F^{i\ov{i}}\vp_{k}\vp_{\ov{k}i\ov{i}}) \\
				\geq {} & \sum_{k}F^{i\ov{i}}\left(|\vp_{ik}|^{2}+|\vp_{i\ov{k}}|^{2}\right)
			 -CK^{\frac{1}{2}}\sum_{k}F^{i\ov{i}}\left(|\vp_{ki}|+|\vp_{k\ov{i}}|\right)-CK\mathcal{F}\\
				\geq {} & \frac{3}{4}\sum_{k}F^{i\ov{i}}\left(|\vp_{ki}|^{2}+|\vp_{k\ov{i}}|^{2}\right)
				-CK\mathcal{F}.
			\end{split}
		\end{equation}
	\end{proof}

	We will use the above computations to prove Proposition \ref{lower bound of L Q}.
	
	\begin{proof}[Proof of Proposition \ref{lower bound of L Q}]
		Combining \eqref{L Q}, \eqref{xieta} and Lemma \ref{lower bound of L lambda1} with Lemma \ref{lower bound of L de u}, we obtain
				\begin{equation}\label{LQ 1}
			\begin{split}
				0\geq \,&
				-\frac{1}{\lambda_1}F^{i\ov{j},k\ov{l}}\nabla_{1}g_{i\ov{j}}\nabla_{\ov{1}}g_{k\ov{l}}-\sum_{i=1}^n\frac{C}{\lambda_1}F^{i\bar i}|g_{1\bar 1, i}|- F^{i\bar{i}}\frac{|(\lambda_1)_{i}|^{2}}{\lambda_1^{2}}\\
				&+\frac{3\eta'}{4} \sum_{j}F^{i\bar{i}}(|\vp_{ij}|^{2}+|\vp_{i\bar j}|^{2})+\eta'' F^{i\bar{i}}|(|\de \vp|^{2})_{i}|^{2} +\xi'L(\vp)\\
				&+\xi'' F^{i\bar{i}}|\vp_{i}|^{2}\
				-C\mathcal{F}.
			\end{split}    
		\end{equation}
		Note that the first term is $G$ and the third term is $N$.
		\end{proof}
		
		Now we complete the proof of Theorem \ref{Thm4.1}.
		
		\begin{proof}
		We now deal with two cases separately. 
		
\medskip
		
		{\bf Case 1.} Assume $\lambda_1\geq -\delta \lambda_n$. Define the set
		\begin{equation}
			I\,=\,\{i:F^{i\bar i}>\delta^{-1}F^{1\bar 1}\}.
			\end{equation} Now we decompose the term N into three terms based on $I$. 
			\begin{equation}
				\begin{split}
				N=\,&F^{i\bar{i}}\frac{|(\lambda_1)_{i}|^{2}}{\lambda_1^{2}}\\
				=\,& \sum_{i\notin I}\frac{F^{i\bar i}|g_{1\bar 1, i}|}{\lambda_1^2}+	2\delta\sum_{i\in I}\frac{F^{i\bar i}|g_{1\bar 1, i}|}{\lambda_1^2}+(1-2\delta)\sum_{i \in I}\frac{F^{i\bar i}|g_{1\bar 1, i}|}{\lambda_1^2}\\
				=:\,&N_1+N_2+N_3.
				\end{split}
				\end{equation}
	Using $dQ(x_0)=0$,  we can prove a lower bound for   the terms $-N_1$ and $-N_2$.
			
			\begin{lemma}\label{third order term 2}
				At $x_0$, we have
				\begin{equation*}
					\begin{split}
					-N_1-N_2 \geq &-\eta''\sum_{k\notin I} F^{i\bar i}||\partial \vp|^2_{i}|^2-2(\xi')^2\delta^{-1}F^{1\bar 1}K-2\delta\eta''\sum_{i\in I} F^{i\bar i}||\partial \vp|^2_{i}|\\
					&-\frac{\xi''}{2}\sum_{k\in I}F^{i\bar i}|\vp_i|^2.
					\end{split}
					\end{equation*}
				\end{lemma}
				\begin{proof}
						Since $Q$ attains its maximum at $x_0$,  we have $dQ(x_0)=0$. Then
					\begin{equation}\label{dQ=0}
						\frac{g_{1\bar 1, i}}{\lambda_1}=-\eta'(|\de \vp|^{2})_{i}-\xi'(\vp)\vp_{i}
					\end{equation}
					for each $1\leq i\leq n$.
							Using \eqref{dQ=0},
		\begin{equation*}
			\begin{split}
			-\sum_{i\notin I}\frac{F^{i\bar i}|g_{1\bar 1, i}|^2}{\lambda_1^2}=\,&-\sum_{i\notin I}F^{i\bar i}|\eta'|\partial \vp|^2_{i}+\xi'\vp_i|^2\\
			\geq\, & -2(\eta')^2\sum_{i\notin I}F^{i\bar i}||\partial \vp|^2_{i}|^2-2(\xi')^2\sum_{i\notin I}F^{i\bar i}|\vp_k|^2\\
			\geq\, &-\eta''\sum_{i\notin I} F^{i\bar i}||\partial \vp|^2_{i}|^2-2(\xi')^2\delta^{-1}F^{1\bar 1}K.
			\end{split}
			\end{equation*}
	Now we choose $\delta$ such that 
\begin{equation}\label{delta}
	4\delta(\xi')^2\leq \frac{1}{2}\xi''.
\end{equation}		
Then, for $k\in I$, using the similar argument, we have
		\begin{equation*}
				\begin{split}
				-2\delta\sum_{i\in I}\frac{F^{i\bar i}|g_{1\bar 1, i}|^2}{\lambda_1^2}
				\geq \,&-2\delta\eta''\sum_{i\in I} F^{i\bar i}||\partial \vp|^2_{i}|^2-4\delta(\xi')^2\sum_{i\in I}F^{i\bar i}|\vp_i|^2\\
				\geq &-2\delta\eta''\sum_{i\in I} F^{i\bar i}||\partial \vp|^2_{i}|^2-\frac{\xi''}{2}\sum_{i\in I}F^{i\bar i}|\vp_i|^2.\\
			\end{split}
			\end{equation*}
	
			\end{proof}

	In the following, we will use $G$ to control the term $N_3$.
				
				\begin{lemma}\label{third order term}
					For any $\epsilon>0$, at $x_0$, we have
					\begin{equation*}
					\begin{split}
					G \geq\,&N_3-\sum_{p}\frac{F^{i\bar i}}{12K}(|\vp_{p\bar i}|^2+|\vp_{pi}|^2)\\
					&+C\xi'F^{i\bar i}|\vp_k|^2+\epsilon C\xi'\mathcal{F}-CF.
					\end{split}
					\end{equation*}
					\end{lemma}
					
			\begin{proof}
				By the concavity of the operator $F$ (see \cite[equation (67)]{Szekelyhidi18}) and the definition of $I$, 
				\begin{equation*}
					\begin{split}
					G\geq \,&\sum_{i\in I}\frac{F^{i\bar i}-F^{1\bar 1}}{\lambda_1-\lambda_i}|\nabla_{1}g_{i\bar 1}|^2\\
						\geq \,& \sum_{i\in I}\frac{(1-\delta)F^{i\bar i}}{\lambda_1-\lambda_i}|\nabla_{1}g_{i\bar 1}|^2\\
						\geq \,& \sum_{i\in I}\frac{(1-2\delta)F^{i\bar i}}{\lambda_1}|\nabla_{1}g_{i\bar 1}|^2.\\
					\end{split}
				\end{equation*}
				Here we used $\delta\lambda_1\geq -\lambda_n$ in the last inequality.			
				
			Next, we compare $g_{i\bar 1,1}$ to $g_{1\bar 1,i}$. By \eqref{commutation formulas}, we have
				\begin{equation*}
					\begin{split}
				g_{i\bar 1,1}\,&=\nabla_{1}(\chi_{i\bar 1}+\vp_{i\bar 1}+W_{i\bar 1,1}+\vp B_{i\bar 1})\\
					\,&= \vp_{i\bar 1 1}+W_{i\bar 1,1}+O(K^{\frac{1}{2}})\\
					\,&= \vp_{1\bar 1 i}+W_{i\bar 1,1}+O(\lambda_1)\\
					\, &=g_{1\bar 1, i}-W_{1\bar 1, i}+W_{i\bar 1,1}+O(\lambda_1).
					\end{split}
					\end{equation*}
				It follows that for any $i$, without summing
				\begin{equation}\label{third order 3}
					\begin{split}
					|g_{i\bar 1,1}|^2\geq &|g_{1\bar 1, i}|^2-C\big(|W_{1\bar 1, i}|^2+|W_{1\bar i,1}|^2+\lambda_1^2+\lambda_1|g_{1\bar 1, i}|\\
					&+|g_{1\bar 1,i}|(|W_{i\bar 1,1}|+|W_{1\bar 1,i}|)\big).
					\end{split}
					\end{equation}
	Now we deal with the terms involving $W$.  We have the following inequalities

				\begin{claim}\label{third order 4}
					At $x_0$, we have	
											\begin{equation*}
								\sum_{i\in I}	\frac{1}{\lambda_1^2}F^{i\bar i}|W_{i\bar 1,1}|^2+|W_{1\bar 1,i}|^2\leq C\mathcal{F}+\frac{C}{\lambda_{1}^2}\sum_{p}F^{i\bar i}|\vp_{ip}|^2.
								\end{equation*}
								and		
				\begin{equation*}
				\begin{split}
					\sum_{i\in I}\frac{1}{\lambda_{1}^2}F^{i\bar i}|g_{1\bar 1, i}|(|W_{i\bar 1,1}|+|W_{1\bar 1 ,i}|)\,\leq & \sum_{p}\frac{F^{i\bar i}}{50K}(|\vp_{p\bar i}|^2+|\vp_{pi}|^2)\\
					&-C_{\epsilon}\xi'F^{i\bar i}|\vp_i|^2-\epsilon C\xi'\mathcal{F}+C\mathcal{F}.
				\end{split}
			\end{equation*}
				\end{claim}
				\begin{proof}
Assume $i\in I$. Then $i\neq 1$ and
					\begin{equation*}
						\nabla_1W_{i\bar 1}=\text{tr}_{\alpha}Z\nabla_{1}\alpha_{i\bar 1}-(n-1)\nabla_{1}Z_{i\bar 1}.
						\end{equation*}
	Set $U=\sum_{p>1, q\geq 1}|u_{pq}|$. 
	By Lemma \ref{prop Z}, it follows that
						\begin{equation}\label{w 2nd order}
							|W_{i\bar 1,1}|+|W_{1\bar 1,i}|\leq C(\lambda_1+U).
							\end{equation}
		Therefore,  by Lemma \ref{prop f}(5),
		\begin{equation}
			\begin{split}
				\frac{1}{\lambda_1^2}F^{i\bar i}(|W_{i\bar 1,1}|^2+|W_{1\bar 1,i}|^2)&\leq C\mathcal{F}(1+\frac{U^2}{\lambda_1^2})\\
				&\leq  C\mathcal{F}+\frac{C}{\lambda_{1}^2}\sum_{p}F^{i\bar i}|\vp_{ip}|^2.
				\end{split}
			\end{equation}						
Using \eqref{dQ=0} and \eqref{xieta},  we have
\begin{equation}\label{dq inequality}
	|g_{1\bar 1,i}|\leq \frac{C\lambda_1}{K^{\frac{1}{2}}}(\sum_{r}|\vp_{\bar r i}|+\sum_{r}|\vp_{ri}|)
	+C\lambda_1|\xi'||\vp_{i}|.
	\end{equation} 
By \eqref{w 2nd order}, 
								\begin{equation*}
									\begin{split}
										\frac{1}{\lambda_{1}^2}F^{i\bar i}|g_{1\bar 1, i}|(|W_{i\bar 1,1}|+|W_{1\bar 1,i}|)
										\leq\,& \frac{CF^{i\bar i}}{\lambda_1 K^{\frac{1}{2}}}(\sum_{r}|\vp_{\bar r i}|+\sum_{r}|\vp_{ri}|)(\lambda_1+U)\\
										&+ \frac{CF^{i\bar i}}{\lambda_1 }|\xi'||\vp_{i}|(\lambda_1+U).
										\end{split}
									\end{equation*}
									Now  we deal with the first term. By directly calculation and Lemma \ref{prop f}(5),
									\begin{equation}
										\begin{split}
											   &\sum_{i\in I}\frac{CF^{i\bar i}}{\lambda_1 K^{\frac{1}{2}}}(\sum_{r}|\vp_{\bar r i}|+\sum_{r}|\vp_{ri}|)(\lambda_1+U)\\
											\leq & \frac{CF^{i\bar i}}{K^{\frac{1}{2}}}\sum_{r}|\vp_{\bar r i}|+\frac{C\mathcal{F}}{K^{\frac{1}{2}}}U+\frac{C\mathcal{F}}{\lambda_1K^{\frac{1}{2}}}U^2\\
											\leq &\frac{F^{i\bar i}}{100K}\sum_{r}(|\vp_{\bar r i}|^2+|\vp_{pi}|^2)+C\mathcal{F}.
											\end{split}
										\end{equation}
							
											Next, by the Cauchy-Schwarz inequality, for any $\epsilon>0$, there exists a constant $C_{\epsilon}$ such that
											\begin{equation}
												\begin{split}
	          &\sum_{i\in I}\frac{CF^{i\bar i}}{\lambda_1 }|\xi'||\vp_{i}|(\lambda_1+U)\\
	 \leq & -\epsilon \xi'\mathcal{F}-C_{\epsilon}\xi'F^{i\bar i}|\vp_i|^2-\frac{1}{\lambda_1}\mathcal{F}\xi'U^2-\frac{1}{\lambda_1}\xi'F^{i\bar i}|\vp_i|^2\\
	 \leq &-\epsilon \xi'\mathcal{F}-C_{\epsilon}\xi'F^{i\bar i}|\vp_i|^2+\frac{1}{100K}\sum_{p}F^{i\bar i}|\vp_{pi}|^2,
													\end{split}
												\end{equation}
													where we  used the fact that $\xi'<0$.
										
							Combining  the above three inequalities, we obtain
							\begin{equation*}
								\begin{split}
									\sum_{i\in I}\frac{1}{\lambda_{1}^2}F^{i\bar i}|g_{1\bar 1, i}|(|W_{i\bar 1,1}|+|W_{1\bar 1 ,i}|)\,\leq & \sum_{p}\frac{F^{i\bar i}}{50K}(|\vp_{p\bar k}|^2+|\vp_{pk}|^2)\\
									&-C_{\epsilon}\xi'F^{i\bar i}|\vp_i|^2-\epsilon C\xi'\mathcal{F}+C\mathcal{F}.
									\end{split}
								\end{equation*}

									\end{proof}
						In the following, we use  Claim \ref{third order 4} to prove Lemma \ref{third order term}.	Using \eqref{dq inequality}, Lemma \ref{prop f}(5) and the Cauchy-Schwarz inequality, we have
							\begin{equation}\label{dQ=0 1}
								\begin{split}
								\sum_{i\in I}\frac{1}{\lambda_1}F^{i\bar i}|g_{1\bar 1,i}|
									\leq & \frac{1}{2K^{\frac{1}{2}}}\sum_{p}F^{i\bar i}|\vp_{i\bar p}|+\frac{1}{2K^{\frac{1}{2}}}\mathcal{F}U-C_{\epsilon}\xi'F^{i\bar i}|\vp_i|^2-\epsilon\xi'\mathcal{F} \\
									\leq & \sum_{p}\frac{F^{i\bar i}}{100K}(|\vp_{p\bar i}|^2+|\vp_{pi}|^2)-C_{\epsilon}\xi'F^{i\bar i}|\vp_i|^2-\epsilon\mathcal{F}\xi'.
								\end{split}
								\end{equation}

	Combining this inequality with  \eqref{third order 3} and Claim \ref{third order 4}, we finally obtain
			\begin{equation*}
				\begin{split}
				\sum_{i\in I}\frac{F^{i\bar i}|\nabla_1g_{i\bar 1}|^2}{\lambda_1^2}\geq & \sum_{i\in I}\frac{F^{i\bar i}|g_{1\bar 1, i}|^2}{\lambda_1^2}-\sum_{p}\frac{F^{i\bar i}}{30K}(|\vp_{p\bar i}|^2+|\vp_{pi}|^2)\\
				&+C_{\epsilon}\xi'F^{i\bar i}|\vp_i|^2+\epsilon C\xi'\mathcal{F}-C\mathcal{F},
				\end{split}
				\end{equation*}
				if we choose $\lambda_1\geq CK$.
				This complete the proof of the Lemma.
						\end{proof}

					Combining  \eqref{LQ0} with Lemma \ref{third order term 2}, Lemma \ref{third order term} and \eqref{dQ=0 1}, it follows
					\begin{equation*}
						\begin{split}
						0\geq &\sum_{k}\frac{F^{i\bar i}}{100K}(|\vp_{pi}|^2+|\vp_{p\bar i}|^2)+\frac{1}{2}\xi''F^{i\bar i}|\vp_{i}|^2+\xi'L(\vp)\\
						&-2(\xi')^2\delta^{-1}F^{1\bar 1}K-C\mathcal{F}+C_{\epsilon}\xi'F^{i\bar i}|\vp_i|^2+\epsilon C\xi'\mathcal{F}.
						\end{split}
						\end{equation*}

				In the following we deal with $\xi'L(\vp)$.
				\begin{lemma}\label{Lvp}
							\begin{equation*}
						\xi'F^{i\bar i}\vp_{i\bar i}\geq \xi'F^{i\bar i}(g_{i\bar i}-\chi_{i\bar i})+C_{\epsilon}\xi'F^{i\bar i}|\vp_i|^2+(\epsilon+M_0A^{\frac{1}{n+1}})\xi'\mathcal{F}.
					\end{equation*}			
					\end{lemma}
					\begin{proof}
						Note
					\begin{equation*}
						\begin{split}
						\xi'L(\vp)=\xi'F^{i\bar i}(g_{i\bar i}-\chi_{i\bar i}-W_{i\bar i}-\vp B_{i\bar i}).
						\end{split}
					\end{equation*}			
					By \eqref{trans}, we have
					\begin{equation*}
						F^{i\bar i}W_{k\bar k}=\sum_{i}\tilde{F}^{i\bar i}Z_{i\bar i}.
						\end{equation*}
	By Lemma \ref{prop Z} and Lemma \ref{prop f}(5), it follows that 
		\begin{equation*}
			|F^{i\bar i}W_{i\bar i}|\leq C\sum_{i\neq k}\tilde{F}^{i\bar i}|\vp_{k}|\leq C_{\epsilon}F^{i\bar i}|\vp_i|^2+\epsilon\mathcal{F}.
			\end{equation*}			
Using \eqref{assum for 2nd}, $\tilde{B}\leq 0$ and Proposition \ref{upper bound u},
	\begin{equation}
		-\xi'\vp F^{i\bar i} B_{i\bar i}=-\xi'\vp \tilde{F}^{i\bar i} \tilde{B}_{i\bar i}\geq -M_{0}A^{\frac{1}{n+1}}|\xi'|\mathcal{F}.
		\end{equation}
	In conclusion, 
		\begin{equation*}
		\xi'F^{i\bar i}\vp_{i\bar i}\geq \xi'F^{i\bar i}(g_{i\bar i}-\chi_{i\bar i})+C_{\epsilon}\xi'F^{i\bar i}|\vp_i|^2+(\epsilon+M_0A^{\frac{1}{n+1}})\xi'\mathcal{F}.
	\end{equation*}			
								\end{proof}
	Therefore,  we have
	\begin{equation}\label{LQ3}
		\begin{split}
		0\geq & F^{1\bar 1}(\frac{\lambda_1^2}{40K}-2(\xi')^2\delta^{-1}K)+(\frac{1}{2}\xi''+C_{\epsilon}\xi')F^{i\bar i}|\vp_{i}|^2\\
		&-C_0\mathcal{F}-\xi'F^{i\bar i}(g_{i\bar i}-\chi_{i\bar i})+(\epsilon+M_0A^{\frac{1}{n+1}})\xi'\mathcal{F}.
		\end{split}
		\end{equation}
		
			\vspace{.2cm}			
By Proposition \ref{properties coefficients},  there is a uniform positive number $\kappa>0$ such that one of two possibilities occurs:

		(a) We have $F^{i\bar i}(\chi_{i\bar i}-g_{i\bar i})>\kappa\mathcal{F}$, then
		\begin{equation*}
			\begin{split}
			0\geq & F^{1\bar 1}(\frac{\lambda_1^2}{40K}-2(\xi')^2\delta^{-1}K)+(\frac{1}{2}\xi''+C_{\epsilon}\xi')F^{i\bar i}|\vp_i|^2-C_0\mathcal{F}\\
			&+(-\kappa+\epsilon+M_0A^{\frac{1}{n+1}})\xi'\mathcal{F}.
			\end{split}
			\end{equation*} 			
We first choose $\epsilon>0$ and $A$ small enough  such that $-\kappa+\epsilon+M_0A^{\frac{1}{n+1}}\leq -\frac{\kappa}{2}$
and choose $D_2$ so large that
 $\frac{1}{2}\xi''+C_{\epsilon}\xi'\geq 0$ which implies
	\begin{equation*}
	\begin{split}
		0\geq & F^{1\bar 1}(\frac{\lambda_1^2}{40K}-2(\xi')^2\delta^{-1}K)-C_0\mathcal{F}-\frac{1}{2}\kappa\xi'\mathcal{F}.
	\end{split}
\end{equation*} 	
By choosing $D_2$ large enough such that 
$-C_0\mathcal{F}-\frac{1}{2}\kappa\xi'\mathcal{F}\geq 0$, we obtain the required upper bound for $\lambda_1$. 
	
							\vspace{.2cm}
					(b) We have $F^{1\bar 1}>\kappa \mathcal{F}$. 
					Choose the constant that is similar to case(a).
				 By \eqref{LQ3},
					\begin{equation*}
						0\geq \kappa\mathcal{F}(\frac{\lambda_1^2}{40K}-2(\xi')^2\delta^{-1}K)-C\mathcal{F}+\epsilon C \xi'\mathcal{F}+C\xi'\mathcal{F}+\xi'F^{i\bar i}g_{i\bar i}.
						\end{equation*} 
				 Since $F^{i\bar i}g_{i\bar i}\leq \mathcal{F}\lambda_1$,  we obtain
					\begin{equation*}
						0\geq \frac{\kappa\lambda_1^2}{40K^2}-C(1+K^{-1}+\lambda_1K^{-1}),
						\end{equation*}
						which implies
						\[\lambda_1\leq CK.\]
					for a uniform constant C. 
						\medskip
						
					{\bf Case 2.} 
				We assume that $\delta\lambda_1<-\lambda_n$ where   $\delta$ is a  fixed constant. Since $F^{n\bar n}\geq \frac{\mathcal{F}}{n}$ and  $\lambda_n^2>\delta^2\lambda_1^2$,
					\begin{equation*}
						\sum_{p}\frac{F^{i\bar i}}{6K}(|\vp_{pi}|^2+|\vp_{p\bar k}|^2)\geq \frac{F^{n\bar n}}{6K}|\vp_{n\bar n}|^2\geq\frac{\delta^2\mathcal{F}}{6n K}\lambda_1^2-C\mathcal{F}.
					\end{equation*}
					Combining with \eqref{LQ0}, 
					\begin{equation*}
						\begin{split}
							0\geq& \frac{-F^{i\bar i}|g_{1\bar 1,i}|^2}{\lambda_1^2}+\frac{\delta^2}{6nK}\mathcal{F}\lambda_1^2+\eta''F^{i\bar i}|\partial_i |\partial \vp|^2|^2+\xi'L(\vp)\\
							&-C(F^{i\bar i}\lambda_1^{-1}|g_{1\bar 1,i}|+\mathcal{F}).
						\end{split}
					\end{equation*}
					By directly calculation, 
					\begin{equation*}
						F^{i\bar i}|\vp_{i\bar i}|\leq C\mathcal{F}\lambda_1
					\end{equation*}
					and 
					\begin{equation*}
						CF^{i\bar i}\lambda_1^{-1}|g_{1\bar1, i}|\leq \frac{1}{2}\frac{F^{i\bar i}|g_{1\bar 1, i}|^2}{\lambda_1^2}+C\mathcal{F}.
					\end{equation*}
					
					Then we obtain 
					\begin{equation}\label{case 21}
						0\geq -\frac{3}{2}\frac{F^{i\bar i}|g_{1\bar 1, i}|^2}{\lambda_1^2}+\frac{\delta^2}{6nK}\mathcal{F}\lambda_1^2+F^{i\bar i}\eta''|\partial_i |\partial \vp|^2|^2-C\mathcal{F}\lambda_1.
					\end{equation}
					
			By \eqref{dQ=0}, it follows
					\begin{equation*}
						\begin{split}
							&\frac{3}{2}\frac{F^{i\bar i}|g_{1\bar 1, i}|^2}{\lambda_1^2}=\frac{3}{2}F^{i\bar i}|\eta'|\partial \vp|_{i}+\xi'\vp_i|^2\\
							\leq & 2F^{i\bar i}\eta'^2|\partial_i|\partial \vp|^2|^2+CF^{i\bar i}(\xi')^2|\vp_i|^2\\
							\leq & F^{i\bar i}\eta''|\partial_i|\partial \vp|^2|^2+C\mathcal{F}K.
						\end{split}
					\end{equation*}
					Assume $\lambda_1\geq K$.
					Returning to \eqref{case 21}, we obtain, 
					\begin{equation*}
						0\geq \frac{\delta^2\lambda_1^2}{6nK}\mathcal{F}-C\lambda_1\mathcal{F}.
					\end{equation*}
					Then we immediately deduce 
					$
					\lambda_1\leq CK.
					$

	\end{proof}

		Using the blow-up argument \cite{Szekelyhidi18}, we have
		$\sup_{M}|\ddbar \vp|\leq C$.	Using Evans-Krylov's theorem \cite{TWWY15, ZZ11}, we have 	$\|\vp\|_{C^{2,\alpha}}\leq C$.	By a standard argument, we can prove the higher order estimates.	Then we have the following Theorem:
	\begin{theorem}\label{higher estimate}
There exists a constant $C$ such that 
	\begin{equation}
		\|\vp\|_{C^{k}}\leq C,
	\end{equation}
	where $C$ depends on $(M, \alpha), h, \omega_0, A$ and $k$.
\end{theorem}

	\section{Proofs of Theorem \ref{main 2}}
	In this section,  we prove Theorem \ref{main 2}. If $B=0$, using \cite{STW},  we are done. Now we assume $B\neq 0$.
	
	We consider the family of equations $(t\in[0,1])$,
		\begin{equation}\label{equation t1}
			\log\frac{(\tilde{\chi}+\frac{1}{n-1}((\Delta \vp)\alpha-\ddbar \vp)+Z+t\vp\tilde{B})^{n}}{\alpha^n}=(n-1)(th+(1-t)h_0+b_t),
	\end{equation}
	where $h_0=\log \frac{\tilde{\chi}^n}{\alpha^n}$ and $b_t$ are  constant. Suppose $\vp$ satisfies the elliptic condition,
	\begin{equation}\label{Elliptic condition t}
		\tilde{\omega}=\tilde{\chi}+\frac{1}{n-1}((\Delta \vp)\alpha-\ddbar \vp)+Z+t\vp\tilde{B}>0
	\end{equation}
	and the normalization condition
	\begin{equation}\label{Normalization condition t}
		\|e^{\vp}\|_{L^{1}} = A.
	\end{equation}
	We define the space  by
	\[
	\mathcal{U}^{2,\alpha}:=\Big\{\phi\in C^{2,\alpha}(M): \tilde{\omega}>0, \,\,\int_{M}e^{\phi}\,\alpha^n=A\,\,
	\Big\}.
	\]
	We shall prove that  \eqref{equation t1} is solvable for any $t\in[0,1]$.  Let $I$ be the set
	\begin{equation*}
		\{ t\in [0,1] ~|~ \text{there exists $(\vp,t)\in B_{1}$ such that $\Phi(\vp,t)=0$} \}.
	\end{equation*}
	Thus,  to prove Theorem \ref{main 2}, it suffices to prove that $I=[0,1]$.   Note that $\vp_{0}=\frac{1}{2}\ln A$ is a solution of  \eqref{equation t1} at  $t=0$.  Hence, we have $0\in I$.  In the following,   we  prove that the set $I$ is both  open and closed.

	\subsection{Openness}
		Suppose that $(\vp_{\hat{t}},b_{\hat{t}})$ satisfies  $(*)_{\hat{t}}$. In the following, we will show that when $t$ is close to $\hat{t}$, there exists a pair $(\vp_{t},b_t)\in C^{\infty}(M)\times\mathbb{R}$ solving $(*)_{t}$.  	
		
		We denote the linearized operator of \eqref{equation t1} at $\vp_{\hat{t}}$ by:
		\[
		L_{\vp_{\hat{t}}}(\psi):=F^{i\bar j}(\psi)_{i\bar j}+2\re(F^{i\bar j}W^k_{i\bar j}(\psi)_{i})+t\tilde{F}^{i\bar j}\tilde{B}_{i\bar j}\psi.
		\]
		Using maximum principle and $\tilde{B}\neq 0$, 
	\begin{equation}\label{ker L1}
		\mathrm{Ker}	(L_{\vp_{\hat{t}}})=\{0\}.
		\end{equation}
		By the Fredholm alternative, $L_{\vp_{\hat{t}}}$ is a bijective map from $C^{2,\alpha}(M)$ to $C^{\alpha}(M)$. Here $1>\alpha>0$. 
Note that	 the tangent space of $\mathcal{U}^{2,\alpha}$ at $\vp_{\hat{t}}$ is given by
		\begin{equation}\label{tangent space 1}
		T_{\vp_{\hat{t}}}\,\mathcal{U}^{2,\alpha}
		:=\Big\{\psi\in C^{2,\alpha}(M):  \int_{M} e^{\vp_{\hat{t}}}\psi\,\alpha^n=0\,\, \Big\}.
		\end{equation}
		Let us consider the map
		\begin{equation*}
			\begin{split}
				\Phi(\vp,b) = {} & \log \frac{(\tilde{\chi}+\frac{1}{n-1}((\Delta \vp)\alpha-\ddbar \vp)+Z+t\vp\tilde{B})^{n}}{\alpha^n}-b.
			\end{split}
		\end{equation*}
		which maps $\mathcal{U}^{2,\alpha}\times\mathbb{R}$ to $C^{\alpha}(M)$. 
		It is clear that the linearized operator of $\Phi$ at $(\vp_{\hat{t}},\hat{t})$ is given by
		\begin{equation}\label{linear operator}
			(L_{\vp_{\hat{t}}}-b): T_{\vp_{\hat{t}}}\,\mathcal{U}^{2,\alpha}\times \mathbb{R}\longrightarrow  C^{\alpha}(M).
		\end{equation}
		
		Note that for any $h\in C^{\alpha}(M)$, there exists a real function $\psi$ on $M$ such that
		\[
		L_{\vp_{\hat{t}}}(\psi)=h
		\]
and a function $u_0$ satisfying  
\begin{equation}\label{solution for 1}
	 L_{\vp_{\hat{t}}}(u_0)=1.
\end{equation}
Using maximum principle, we have $u_0\leq 0$ and $u_0\neq 0$. Then $\int_{M}u_0e^{\vp_{\hat{t}}}\,\alpha^n<0$.
We can choose $b_0$ such that $\psi+b_0u_0\in T_{\vp_{\hat{t}}}\,\mathcal{U}^{2,\alpha},$ which implies
\[	L_{\vp_{\hat{t}}}(\psi+b_0u_0)-b_0=h.\]
		Hence, the map $L_{u_{\hat{t}}}-b$ is surjective.
		
		 On the other hand, suppose that there are two pairs $(\psi_{1},b_{1}),(\psi_{2},b_{2})\in T_{u_{\hat{t}}}\,\mathcal{U}^{2,\alpha}\times \mathbb{R}$ such that
		\[
		L_{\vp_{\hat{t}}}(\psi_{1})-b_{1}
		= L_{\vp_{\hat{t}}}(\psi_{2})-b_{2}.
		\]
		It then follows that
		\[
		L_{\vp_{\hat{t}}}(\psi_{1}-\psi_{2}) = b_{1}-b_{2}.
		\]
		Since $L_{\vp_{\hat{t}}}$ is bijective,
		by \eqref{solution for 1}, we have
		$\psi_{1}-\psi_{2}=(b_1-b_2) u_0$. Since $$\int_{M}(\psi_{1}-\psi_{2})e^{\vp_{\hat t}}\alpha^n=0,$$
		 We have $b_1=b_2$ which implies that $\psi_{1}=\psi_2$.
		 Then $L_{u_{\hat{t}}}-b$ is injective.
		
		Now we conclude that $L_{u_{\hat{t}}}-c$ is bijective. By the inverse function theorem, when $t$ is close to $\hat{t}$, there exists a pair $(\vp_{t},b_t)\in\mathcal{U}^{2,\alpha}\times\mathbb{R}$ satisfying
		\begin{equation*}
			\Phi(\vp_{t},b_t) = th+(1-t)h_{0}.
		\end{equation*}
		The standard elliptic theory shows that $\vp_{t}\in C^{\infty}(M)$. Then $I$ is open.

		\subsection{Closeness}
		Since $0\in I$ and $I$ is open, there exists $t_{0}\in (0,1]$ such that $[0,t_{0})\subset I$. We need to prove $t_{0}\in I$. First we  show that $\{b_{t}\}$ is uniformly bounded.
		\begin{claim}\label{bt}
			There exists a constant $C$ such that 
			\begin{equation}\label{Closeness equation 1}
				|b_t|\leq C.
			\end{equation}
		\end{claim}
		\begin{proof}
			Suppose that $\vp_{t}(p_t)=\max_{M}\vp_t$ which implies$\ddbar \vp_{t}\leq 0$ at $p_{t}$.
			Using the equation \eqref{n-1 ma}, we have
			\begin{equation}
				\begin{split}
					th(p_t)+(1-t)h_0(p_t)+b_t\leq& \frac{(\tilde{\chi}+\vp_{t}(p_t)\tilde{B})^n}{(n-1)\alpha^n}\\
					\leq &\frac{(\tilde{\chi}+\min\{0, \vp_{t}(p_t)\} \tilde{B})^n}{(n-1)\alpha^n}.
				\end{split}
			\end{equation}
			Using the normalized condition $\int_{M}e^{\vp}\alpha^n=A$, we have $\vp_{t}(p_t)\geq \ln A$.
			Then we can obtain the upper bound of $b_{t}$:
			\begin{equation}
				\begin{split}
					b_t\leq& \frac{(\tilde{\chi}+\vp_{t}(p_t)\tilde{B})^n}{(n-1)\alpha^n}\\
					\leq &\frac{(\tilde{\chi}+\min\{0, \ln A\} \tilde{B})^n}{(n-1)\alpha^n}-th(p_t)-(1-t)h_0(p_t).
				\end{split}
			\end{equation}
			The lower bound of $b_t$ can be proved similarly by considering the minimum point.
		\end{proof}

Now we  prove the zero order estimate.   In fact,  we have
		\begin{proposition}\label{claim-2}
			There exists a constant $A_0$ such that  if $A\leq A_0$,
			\begin{equation}\label{zero Closeness equation 1}
			\sup\vp \leq 2M_{0}A^{\frac{1}{n+1}}, ~t\in[0,t_{0}),
			\end{equation}
			where $M_{0}$ is the constant in Proposition \ref{upper bound u}.
		\end{proposition}
		\begin{proof}
		Note that  $\vp_{0}=\ln A$.  Then $\sup_{M}\vp_0\leq 2M_{0}A^{\frac{1}{n+1}}$, which satisfies (\ref{zero Closeness equation 1}).   Thus,  if (\ref{zero Closeness equation 1}) is false,  there will  exist  $\tilde{t}\in (0,t_{0})$ such that
		\begin{equation}\label{Closeness equation 2}
			\sup_{M}\vp_{\tilde{t}} = 2M_{0}A^{\frac{1}{n+1}}.
		\end{equation}
		We may assume  that $2M_{0}A^{\frac{1}{n+1}}\leq   1$.    Hence,  we can  apply  Proposition \ref{upper bound u} to $\vp_{\ti{t}}$  and  we obtain
		\begin{equation*}
			\sup_{M}\vp_{\tilde{t}}  \leq M_{0}A,
		\end{equation*}
		which contradicts to (\ref{Closeness equation 2}). This proves (\ref{zero Closeness equation 1}).\\
		\end{proof}
		
Combining Proposition \ref{claim-2} with Proposition \ref{Prop32}, we
have the zero order estimates and the second order estimates Theorem \ref{Thm4.1}. 
Then $C^{\infty}$ a priori estimates of $\vp_{t}$ follows from Theorem \ref{higher estimate}. Combining this with the Arzel\`a-Ascoli theorem, $I$ is closed.
	
		Hence we prove the existence of solutions for the equation \eqref{Herm n-1}. 
		\vspace{2mm}
		
		Now we prove uniqueness.  Assume that we have two solutions $(\vp, b)$ and $(\vp', b')$. Without loss of generality, we can assume there exists a point such that $\vp(p)> \vp'(p)$. Then we have $\inf_{M} (\vp-\vp')\leq 0$. If not, then $\vp\geq \vp'$. Using  $\vp(p)> \vp'(p)$ and $\int_{M}e^{\vp}\alpha^{n}=\int_{M}e^{\vp'}\alpha^{n}$, it is a contraction. Note
		\begin{equation}
		\frac{	(\tilde{\omega}_{\vp'}+\frac{1}{n-1}((\Delta \theta)\alpha-\ddbar \theta)+Z(\partial \theta)+\theta\tilde{B})^{n}}{	(\tilde{\omega}_{\vp'})^{n}}=e^{(n-1)(b-b')}.
			\end{equation}
			where $\tilde{\omega}_{\vp'}=\tilde{\chi}+\frac{1}{n-1}((\Delta \vp')\alpha-\ddbar \vp')+Z(\partial \vp')+\vp'\tilde{B}$ and $\theta=\vp-\vp'$.
	Consider the maximum point of $\vp-\vp'$, we can prove $b\geq b'$. Similarly, it follows that $b'\geq b$ and so $b=b'$.  Then $F(\omega_{\vp})-F(\omega_{\vp'})=0$, where $\omega_{\vp}=T(\tilde{\omega}_{\vp})$.
Using the maximum principle, we can prove $\vp=\vp'$.

\end{document}